\documentclass[12pt]{amsart}
\usepackage{amsmath,amsthm,amsfonts,amssymb,amscd,euscript}

\newlength{\defbaselineskip}
\theoremstyle{plain}

\theoremstyle{definition}

\newtheorem{theoremalpha}{Theorem}
\newtheorem{corollaryalpha}[theoremalpha]{Corollary}

\theoremstyle{plain}
\newtheorem{thm}{Theorem}

\newtheorem{lem}[thm]{Lemma}
\newtheorem{cor}[thm]{Corollary}

\newtheorem{prop}[thm]{Proposition}
\newtheorem{conj}[thm]{Conjecture}

\theoremstyle{definition}
\newtheorem{defn}[thm]{Definition}

\newtheorem{exmp}[thm]{Example}

\newtheorem{rem}[thm]{Remark}
\newtheorem{remark}[thm]{Remark}
\numberwithin{equation}{section}

\def\C{{\mathbb C}}
\def\D{{\mathfrak D}}
\def\N{{\mathbb N}}
\def\Q{{\mathbb Q}}
\def\Z{{\mathbb Z}}
\def\w{{\mathrm w}}

\def\bmx{\begin{bmatrix}}
\def\emx{\end{bmatrix}}
\def\LM{{\rm\scriptstyle{LM}}}
\def\LT{{\rm\scriptstyle{LT}}}

\newcommand{\boxs}[1]
{ \multiput(#1)(10,0){2}
 {\line(0,10){10}}
\multiput(#1)(0,10){2}
 {\line(10,0){10}}
}

\address{Department of Mathematics, Purdue University, West Lafayette, IN 47907}
\email{{\tt kyungl@purdue.edu}}
\thanks{Research of the first author partially supported by NSF grant DMS
0901367}

\address{Department of Mathematics, University of Illinois at Urbana Champaign, Urbana, IL 61801}
\email{{\tt llpku@math.uiuc.edu}}

\usepackage{hyperref}

\begin{document}
\title[Diagonal ideals]{$q,t$-Catalan numbers and generators for the radical ideal defining the diagonal locus of $(\C^2)^n$}
\author{Kyungyong Lee and Li Li}
\maketitle

\begin{abstract}
Let $I$ be the ideal generated by alternating polynomials in two
sets of $n$ variables. Haiman proved that the $q,t$-Catalan number
is the Hilbert series of the graded vector space
$M(=\bigoplus_{d_1,d_2}M_{d_1,d_2})$ spanned by a minimal set of
generators for $I$. In this paper we give simple upper bounds on
$\text{dim }M_{d_1, d_2}$ in terms of partition numbers, and find
all bi-degrees $(d_1,d_2)$ such that $\dim M_{d_1, d_2}$
achieve the upper bounds. For such bi-degrees, we also find explicit bases
for $M_{d_1, d_2}$. The main idea is to define and study a nontrivial linear map from $M$ to a polynomial ring $\C[\rho_1, \rho_2,\dots]$.
\end{abstract}

\tableofcontents

\section{Introduction}
In \cite{GHai:qtCatalan}, Garsia and Haiman introduced the
$q,t$-Catalan number $C_n(q,t)$, and they showed that $C_n(q,1)$
agrees with the $q$-Catalan number defined by Carlitz and Riordan
\cite{CR}. To be more precise, take the $n\times n$ square whose
southwest corner is $(0,0)$ and northeast corner is $(n,n)$. Let
$\mathcal{D}_n$ be the collection of Dyck paths, i.e. lattice paths
from $(0, 0)$  to $(n,n)$ that proceed by NORTH or EAST steps and
never go below the diagonal. For any Dyck path $\Pi$, let
$\text{area}(\Pi)$ be the number of lattice squares below $\Pi$ and
strongly above the diagonal. Then
$$
C_n(q,1)=\sum_{\Pi\in \mathcal{D}_n} q^{\text{area}(\Pi)}.
$$

The $q,t$-Catalan number $C_n(q,t)$ also has a combinatorial interpretation using Dyck paths.
Given a Dyck path $\Pi$, let $a_i(\Pi)$ be the number of
squares in the $i$-th row that lie in the region bounded by $\Pi$ and the diagonal.
Garsia and Haglund
(\cite{GH:pos}, \cite{GH:proof}) among others showed that
$$
C_n(q,t)=\sum_{\Pi\in \mathcal{D}_n}
q^{\text{area}(\Pi)}t^{\text{dinv}(\Pi)},
$$
where  $$\text{dinv}(\Pi):=|\{(i,j)\,|\,i<j\text{ and
}a_i(\Pi)=a_j(\Pi)   \}| \, + \,  |\{(i,j)\,|\,i<j\text{ and
}a_i(\Pi)+1=a_j(\Pi) \}|.$$

A very natural question is to find the coefficient of $q^{d_1}
t^{d_2}$ in $C_n(q,t)$ for each pair $(d_1,d_2)$, in other words, to
count how many Dyck paths have the same statistics (area, dinv). It
is well-known that the sum $\text{area}(\Pi)+\text{dinv}(\Pi)$ is at
most ${n\choose 2}$. In this paper we find coefficients of
$q^{d_1}t^{d_2}$ when ${n\choose 2}-d_1-d_2$ is relatively small.

Denote by $p(k)$ the partition number of $k$ and by convention
$p(0)=1$ and $p(k)=0$ for $k<0$. Denote by $p(b,k)$ the partition
number of $k$ into no more than $b$ parts, and by convention
$p(0,k)=0$ for $k>0$, $p(b,0)=1$ for $b\ge 0$. One of our main
results is as follows.

\begin{theoremalpha}\label{coeff} Let $d_1,d_2$ be non-negative integers $d_1, d_2$ with $d_1+d_2\leq {n\choose 2}$.
Define  $k={n\choose 2}-d_1-d_2$ and $\delta=\min(d_1,d_2)$. Then the
coefficient of $q^{d_1}t^{d_2}$ in $C_n(q,t)$  is less than or equal
to $p(\delta,k)$, and the equality holds
if and only if one the following conditions holds:
\begin{itemize}
\item $k\leq n-3$, or
\item $k=n-2$ and $\delta=1$, or
\item $\delta=0$.
\end{itemize}
\end{theoremalpha}

This theorem is a consequence of Theorem C. It  contains
\cite[Theorem 6]{LL} and a result of Bergeron and Chen
\cite[Corollary 8.3.1]{BC} as special cases. In fact it proves
\cite[Conjecture 8]{LL}. We feel that the coefficient of
$q^{d_1}t^{d_2}$ for general $k$ can also be expressed in terms of
partition numbers, only that the expression might be complicated.
For example, we give the following conjecture which is verified for
$6\le n\le 10$.

\noindent\textbf{Conjecture.} Let $n, d_1, d_2, \delta, k$ be as in
Theorem \ref{coeff}.  If $n-2\leq k\leq 2n-8$ and $\delta\ge k$,
then the coefficient of $q^{d_1}t^{d_2}$ equals $$p(k)
-2[p(0)+p(1)+\cdots+p(k-n+1)]-p(k-n+2).$$

\medskip
 As a corollary of Theorem A, we can compute some higher
degree terms of the specialization at $t=q$.

\begin{corollaryalpha}
$$C_n(q,q)=\sum_{k=0}^{n-3} \left(p(k)\left({n\choose2}-3k+1\right)+2\sum_{i=1}^{k-1} p(i,k)\right)q^{{n\choose2}-k} + (\text{lower degree terms}).
$$
\end{corollaryalpha}

From the perspective of commutative algebra, the $q,t$-Catalan number
is closely related to the graded ideal $I$ defining the diagonal
locus of $(\C^2)^n$. In \cite{H:hil} and \cite{van}, Haiman proved
that the $q,t$-Catalan number is the Hilbert series of the graded
vector space spanned by minimal generators for $I$. Blowing up the
ideal $I$ gives the well-known isospectral Hilbert scheme discovered
by Haiman in his proof of the $n!$ conjecture and the
positivity conjecture for the Kostka-Macdonald coefficients \cite{H:hil}. A
natural question, posed by Haiman \cite{H04}, is to study a minimal set of
generators of the ideal $I$. An extensive study of
generators of $I$ might lead to an explicit principalization of the
ideal $I$.

To construct a minimal set of generators of $I$ is difficult.
However, if we focus on cases when the degree is ${n\choose 2}-k$
where $k\le n-3$, we can give an explicit combinatorial description
for a minimal set of generators.

Now we turn to a detailed description. Fix a positive integer $n$.
Consider $n$-tuples of ordered points $\{(x_i,y_i)\}_{1\le i\le n}$
in the plane $\C^2$. The set of all $n$-tuples forms an affine space
$(\C^2)^n$ with coordinate ring
$\C[\textbf{x},\textbf{y}]=\C[x_1,y_1,...,x_n,y_n]$. Denote by   $\mathbb{C}[\textbf{x},\textbf{y}]^{\epsilon}$ the
vector space of alternating polynomials spanned by a basis $\{\Delta(D)\}_{D\in\D_n}$ defined as follows. Denote
by $\N$ the set of nonnegative integers. Let $\D_n$ be the set of
subsets $D=\{(\alpha_1,\beta_1),...,(\alpha_n,\beta_n)\}$ of
$\N\times\N$. For $D\in\D_n$, define
$$\Delta(D):=
\det\begin{bmatrix}
     x_1^{\alpha_1}y_1^{\beta_1}&x_1^{\alpha_2}y_1^{\beta_2}&...&x_1^{\alpha_n}y_1^{\beta_n}\\
      \vdots&\vdots &\ddots &\vdots\\
    x_n^{\alpha_1}y_n^{\beta_1}&x_n^{\alpha_2}y_n^{\beta_2}&...&x_n^{\alpha_n}y_n^{\beta_n}\\
   \end{bmatrix}.
$$
The ideal $I\subset \C[\textbf{x},\textbf{y}]$ is the radical ideal
that defines the locus where at least two points coincide, to be
precise,
$$I=\bigcap_{1\leq i<j\leq n}(x_i-x_j, y_i-y_j).$$
Haiman \cite{H:hil}  has proved that $I$ is in fact generated by
$\mathbb{C}[\textbf{x},\textbf{y}]^{\epsilon}$, therefore is
generated by $\{\Delta(D)\}_{D\in \D_n}$.

Finding a minimal set of generators of $I$ is equivalent to finding
a basis of $M:=I/(\mathbf{x},\mathbf{y})I$ where
$(\mathbf{x},\mathbf{y})$ is the maximal ideal
$(x_1,y_1,\dots,x_n,y_n)$. Since $M$ is naturally bi-graded with respect
to $x$-degree and $y$-degree, we can write
$$M=\bigoplus_{d_1,d_2}M_{d_1,d_2}.$$
In \cite[p393]{van}, Haiman discovered the amazing fact that
\begin{equation}\label{qtcathil}C_n(q,t)=\sum_{d_1,d_2}t^{d_1}q^{d_2}\dim M_{d_1,d_2}.\end{equation}
Setting $q=t=1$, we get
$$\dim_{\C}M=
\displaystyle\frac{1}{n+1}{2n\choose n},$$ which is the usual Catalan number $C_n$.

The authors showed in \cite{LL} that, when the
deficit $k={n\choose 2}-d_1-d_2$ is relatively small compared to
$n$ and $d_1, d_2$ are not too small,  an explicit basis of $M_{d_1,d_2}$ can be constructed in one-to-one
correspondence with partitions of $k$, by using what we call \emph{minimal
staircase forms}. However the bound of $k$ given in \cite{LL} was by
no means sharp. In this paper we find all bi-degrees $(d_1,d_2)$ for
which $\text{dim }M_{d_1, d_2}$ are exactly partition numbers of $k$
into no more than $\min(d_1,d_2)$ parts. For such bi-degrees, we
also find bases for $M_{d_1, d_2}$.

\begin{theoremalpha}\label{bases}

Let $d_1,d_2$ be non-negative integers $d_1, d_2$ with $d_1+d_2\leq {n\choose 2}$. Define  $k={n\choose 2}-d_1-d_2$ and $\delta=\min(d_1,d_2)$. Then $\dim M_{d_1, d_2}\le p(\delta,k)$, and the equality holds
if and only if one the following conditions holds:
\begin{itemize}
\item $k\leq n-3$, or
\item $k=n-2$ and $\delta=1$, or
\item $\delta=0$.
\end{itemize}
In case the equality holds, there is an
explicit construction of a basis of $M_{d_1,d_2}$.
\end{theoremalpha}

The theorem is a consequence of Theorem \ref{main:k<=n-3} and
Theorem \ref{thm:equality holds}. Obviously, Theorem~\ref{coeff}
immediately follows from Theorem~\ref{bases}, thanks to
(\ref{qtcathil}), a theorem of Haiman. The idea of the construction
consists of two parts: the easier part is to show  $$\dim M_{d_1,
d_2}\le p(\delta, k)$$ using a new characterization of $q,t$-Catalan
numbers; the harder part is to construct a set of $p(\delta, k)$
linearly independent elements in $M_{d_1,d_2}$. It seems difficult
(as least to the authors)
 to test directly whether a given set of elements in $M_{d_1,d_2}$ are linearly independent. We define a map
$\varphi$ sending an alternating polynomial $f\in
\mathbb{C}[\textbf{x},\textbf{y}]^{\epsilon}$ to a polynomial ring
$\C[\rho_1,\rho_2,\rho_3,\dots]$. The map has two desirable
properties: (i) for many $f$, $\varphi(f)$ can be easily computed,
and (ii) for each bi-degree $(d_1,d_2)$, $\varphi$ induces a
morphism $\bar{\varphi} : M_{d_1,d_2} \longrightarrow \C[\rho_1,
\rho_2,...]$ of $\C$-modules. Then we use the fact the linear
dependency is easier to check in $\C[\rho_1, \rho_2,...]$ than in
$M_{d_1,d_2}$. This idea is motivated by our earlier work \cite{LL}.

The structure of the paper is as follows. After introducing some
notations in \S2, we define and study the map $\varphi$ in \S3, then
in \S4 and \S5 we give the upper bound and the lower bound of $\dim
M_{d_1,d_2}$, and prove the main result in \S6. For readers'
convenience, we give the table of $q,t$-Catalan numbers for $n=7$ in
appendix \S7.1 and a Macaulay 2 code for computing the map $\varphi$
in \S7.2.

\noindent \emph{Acknowledgements. }We are grateful to Fran\c{c}ois Bergeron, Mahir Can, Jim Haglund, Nick Loehr, Alex Woo and Alex
Yong for valuable discussions and correspondence.

\section{Notation}
\begin{itemize}
\item We adopt the convention that $\N$ is the set of natural numbers \emph{including}
zero, and $\N^+$ is the set of positive integers.
\smallskip

\item For $n\in \N^+$, define $\D_n=\{D=\{(a_1,b_1),\dots,(a_n,b_n)\}|
a_i,b_i\in\N\}$, i.e. an element of $\D_n$ is an ordered set $D$ of
$n$ points in $\N\times\N$. Define $\D=\cup_{n=1}^\infty\D_n$.
Similarly, define $\D'_n=\{D=\{(a_1,b_1),\dots,(a_n,b_n)\}|
a_i\in\Z,b_i\in\N, a_i+b_i\ge 0\}$. Define
$\D'=\cup_{n=1}^\infty\D'_n$.

We use $P_i$ to denote the point $(a_i,b_i)$, and denote
$|P_i|=a_i+b_i$, $|P_i|_x=a_i$, $|P_i|_y=b_i$.

Unless otherwise specified, we assume throughout the paper that
\begin{equation}\label{order}
P_1<P_2<\cdots<P_n,\quad\hbox{ for } D=\{P_1,\dots,P_n\}
\end{equation}
where the order is defined as follows:
$$ \textrm{ $(a,b)<(a',b')$
if $a+b<a'+b'$, or if $a+b=a'+b'$ and $a<a'$.}$$
In particular,
$|P_1|\le|P_2|\le\cdots\le|P_n|$.

\smallskip

\item Given a monomial $f=x_1^{\alpha_1}y_1^{\beta_1}\cdots
x_n^{\alpha_n}y_n^{\beta_n}\in\C[\textbf{x},\textbf{y}]$, we call
$(\sum_{i=1}^n \alpha_i,\sum_{i=1}^n \beta_i)$ the bi-degree of $f$.
A polynomial in $\C[\textbf{x},\textbf{y}]$ is bi-homogeneous of
bi-degree $(d_1,d_2)$ if all its monomials have the same bi-degree
$(d_1,d_2)$.

Given $D=\{(a_1,b_1),\dots,(a_n,b_n)\}\in\D_n$, we call
$(\sum_{i=1}^n a_i,\sum_{i=1}^n b_i)$ the bi-degree of $D$, which is
the same as the bi-degree of the polynomial $\Delta(D)$.

\smallskip

\item Let $k,b\in \N^+$. Denote the set of partitions of $k$ as
$$\Pi_k=\{\nu=(\nu_1,\nu_2,\cdots)| \nu_i\in\N^+,
\nu_1\le\nu_2\le\cdots, \hbox{ and } \nu_1+\nu_2+\cdots=k)\}.$$
Denote by $\Pi_{b,k}$ the set of partitions of $k$ into at most $b$
parts.

Define the partition numbers $p(k)=\#\Pi_k$ and
$p(b,k)=\#\Pi_{b,k}$. By convention $p(0)=0$, $p(0,k)=0$ for $k>0$, $p(b,0)=1$ for all $b\ge 0$.

\smallskip

\item Let $\Z[\rho]=\Z[\rho_1,\rho_2,\dots]$ be the polynomial ring with
countably many variables $\rho_1,\rho_2,\dots$. By convention we
assume $\rho_0=1$. For a partition
$\nu=(\nu_1,\nu_2,\cdots)\in\Pi_k$, define
$\rho_\nu=\rho_{\nu_1}\rho_{\nu_2}\cdots\in \Z[\rho]$.
\smallskip

\item For $n\in\N^+$, denote by $S_n$ the permutation group of $\{1,...,n\}$.
\smallskip

\item Given two bi-homogeneous polynomial $f,g$ of bi-degree $(d_1,d_2)$, let $\bar{f},\bar{g}$ be the corresponding element
in $M_{d_1,d_2}$. We say that $f\equiv g$ (modulo lower
degrees) if $\bar{f}=\bar{g}$ in $M_{d_1,d_2}$.

\end{itemize}

\section{Map $\varphi$.}

\subsection{Definition and properties of $\varphi$}
In this subsection we define and study the map $\varphi$ which
naturally arises when we look for a minimal set of generators of the
ideal $I$ of alternating polynomials. For readers' convenience, a
Macaulay 2 code for computing $\varphi$ is put in Appendix.
\begin{defn}\label{linearcompar}
(a) Define the map $ \varphi:\D'_n \to \mathbb{Z}[\rho]$ as follows.
Let $D=\{(a_1,b_1),...,(a_n,b_n)\}\in\D'_n$, $k={n\choose
2}-\sum_{i=1}^n (a_i+b_i)$, and define
$$\varphi(D)
:=(-1)^{k} \sum_{\sigma\in S_n} \text{sgn}(\sigma) \prod_{i=1}^n
\left(\sum \rho_{w_1}\rho_{w_2}\cdots\rho_{w_{b_i}}\right), $$ where
$(w_1,\dots,w_{b_i})$ in the sum $\sum
\rho_{w_1}\rho_{w_2}\cdots\rho_{w_{b_i}}$ runs through the set
\begin{equation}\label{varphi condition}\{(w_1,\dots,w_{b_i})\in
\N^{b_i} | \; w_1+...+w_{b_i}= \sigma(i)-1-a_i-b_i\},\end{equation}
with the convention that
$$
\sum \rho_{w_1} ...\rho_{w_{b_i}}=\left\{\begin{array}{ll}
        0 & \textrm{ if  $\sigma(i)-1-a_i-b_i<0$}; \\
        0 & \textrm{ if  $b_i=0$ and $\sigma(i)-1-a_i-b_i>0$}; \\
        1 & \textrm{ if  $b_i=0$ and $\sigma(i)-1-a_i-b_i=0$}.
      \end{array}
 \right.
$$

(b) Here is an equivalent definition of $\varphi(D)$. Define the
weight of $\rho_i$ to be $i$ for $i\in\N^+$ and define the weight of $\rho_0=1$ to be $0$. Naturally the weight of any monomial $c\rho_{i_1}...\rho_{i_n}$ $(c\in\Z)$ is defined to be $i_1+...+i_n$.  For $\w\in\N$ and a power series $f\in
\Z[[\rho_1,\rho_2,\dots]]$, denote by $\{f\}_\w$ the sum of terms of weight-$\w$
in $f$, which is a polynomial. Define
$$h(b,\w):=\big{\{}(1+\rho_1+\rho_2+\cdots)^b\big{\}}_\w, \quad b\in\N,\w\in\Z.$$
Naturally $h(b,\w)=0$ if $\w<0$. Also assume
$(1+\rho_1+\rho_2+\cdots)^0=1$. Then
$$\varphi(D)=(-1)^k\begin{vmatrix}
h(b_1, -|P_1|)& h(b_1,1-|P_1|)& h(b_1,2-|P_1|)&\cdots&h(b_1,n-1-|P_1|)\\
h(b_2,-|P_2|)& h(b_2,1-|P_2|)& h(b_2,2-|P_2|)&\cdots&h(b_2,n-1-|P_2|)\\
\vdots&\vdots&\vdots&\ddots&\vdots\\
h(b_n,-|P_n|)& h(b_n,1-|P_n|)& h(b_n,2-|P_n|)&\cdots&h(b_n,n-1-|P_n|)\\
\end{vmatrix}.$$

(c) Let $D_1,\dots,D_\ell\in D'$ be of the same bi-degree and
$\sum_{i=1}^\ell c_iD_i$ be the formal sum for any $c_i\in\C$ ($1\le
i\le \ell$). Define
$$\varphi(\sum_{i=1}^\ell c_iD_i):=\sum_{i=1}^\ell c_i\,\varphi(D_i).$$
For any bi-homogeneous alternating polynomials $f=\sum_{i=1}^\ell
c_i\,\Delta(D_i)\in \C[\textbf{x},\textbf{y}]^\epsilon$, we define
$$\varphi(f):=\varphi(\sum_{i=1}^\ell c_iD_i)=\sum_{i=1}^\ell c_i\,\varphi(D_i)$$
by abuse of notation. \qed
\end{defn}

Before relating $\varphi(D)$ with $\Delta(D)$, we shall first look
at some properties of the map $\varphi$.

\begin{lem}\label{varphi_computation_lemma}
Let $n\in\N^+$, $D=\{P_1,\dots,P_n\}\in\D'_n$ where $P_1<...<P_n$ as
in the assumption (\ref{order}).

{\rm (i)} If $|P_i|\ge i$ for some $1\le i\le n$, then
$\varphi(D)=0$.

{\rm (ii)} Let $m\in\N^+$ and $Q_1,\dots,Q_m\in \Z\times\N$ satisfy
$|Q_i|=i-1$ for $1\le i\le m$. Let
$$\tilde{D}=\{Q_1,\dots,Q_m,P_1+(m,0),P_2+(m,0),\dots,P_n+(m,0)\}.$$
Then $\varphi(\tilde{D})=\varphi(D)$.

{\rm (iii)} Let $t\in\N^+$, $Q=(-t,t)$ and
$\tilde{D}=\{P_1+Q,P_2+Q,\dots,P_n+Q\}$. Then
$$\varphi(\tilde{D})=\varphi(D).$$

{\rm (iv)} Let $S=\{i\,\big{|}\;|P_i|:=a_i+b_i=i-1
\}=\{i_1<\cdots<i_\ell \}$ and assume $i_1=1$. We define the set
$\{P_{i_r},\dots,P_{i_{r+1}-1}\}$ \emph{the $r$-th block} of $D$ for
$1\le r\le \ell$ (assuming $P_{i_{\ell+1}}=n+1$). Then
$$\varphi(D)=\prod_{r=1}^\ell \varphi( \{
P_{i_r}-(i_r-1,0),P_{i_r+1}-(i_r-1,0),\dots,P_{i_{r+1}-1}-(i_r-1,0)\}
).$$

{\rm (v)} Suppose $|P_i|=0$ for $1\le i\le n$. Then
$\varphi(D)=c\cdot\rho_1^{{{n}\choose 2}}$ for a positive integer
$c$. In fact, $$c=\frac{\prod_{i<j}(b_i-b_j)}{1!2!\cdots(n-1)!}.$$

{\rm (vi)} For $s\in\N^+$, let
$D=\{(-1,1),(0,0),(1,0),\dots,(s-1,0))\}$. Then
$$\varphi(D)=\rho_s.$$
\end{lem}

Before giving the proof, let us look at some examples explaining the
lemma.

\begin{exmp}
(i) We have $\varphi(\{(0,0),(1,0),(2,1),(3,0)\})=0$ since $|P_3|=2+1=3$.

(ii) Let $D=\{(-1,1),(0,0),(0,1)\}$, $m=2$,
$Q_1=(0,0),Q_2=(1,0)$. Then
$$\varphi(\{(0,0),(1,0),(1,1),(2,0),(2,1)\})=\varphi(D),$$
i.e.
$$
\varphi\big{(}\hspace{16pt} \begin{picture}(27,25)
\put(-3,-3){$\bullet$}\put(7,-3){$\bullet$}\put(7,7){$\bullet$}\put(17,-3){$\bullet$}
\put(17,7){$\bullet$}
\boxs{-10,0}\boxs{0,0}\boxs{10,0}
\boxs{-10,10}\boxs{0,10}\boxs{10,10}
\linethickness{1pt}\put(0,0){\line(0,1){27}}
\linethickness{1pt}\put(-15,0){\line(1,0){40}}
\end{picture} \big{)}
= 
\varphi\big{(}\hspace{16pt} \begin{picture}(27,25)
\put(-13,7){$\bullet$}\put(-3,-3){$\bullet$}\put(-3,7){$\bullet$}
\boxs{-10,0}\boxs{0,0}\boxs{10,0}
\boxs{-10,10}\boxs{0,10}\boxs{10,10}
\linethickness{1pt}\put(0,0){\line(0,1){27}}
\linethickness{1pt}\put(-15,0){\line(1,0){40}}
\end{picture} \big{)}.$$

(iii) Let $D=\{(0,0),(0,1),(1,0)\}$, $t=1$. Then
$$\varphi\big{(}\hspace{16pt} \begin{picture}(27,25)
\put(-13,7){$\bullet$}\put(-13,17){$\bullet$}\put(-3,7){$\bullet$}
\boxs{-10,0}\boxs{0,0}\boxs{10,0}
\boxs{-10,10}\boxs{0,10}\boxs{10,10}
\linethickness{1pt}\put(0,0){\line(0,1){27}}
\linethickness{1pt}\put(-15,0){\line(1,0){40}}
\end{picture} \big{)}
= 
\varphi\big{(}\hspace{16pt} \begin{picture}(27,25)
\put(-3,-3){$\bullet$}\put(-3,7){$\bullet$}\put(7,-3){$\bullet$}
\boxs{-10,0}\boxs{0,0}\boxs{10,0}
\boxs{-10,10}\boxs{0,10}\boxs{10,10}
\linethickness{1pt}\put(0,0){\line(0,1){27}}
\linethickness{1pt}\put(-15,0){\line(1,0){40}}
\end{picture} \big{)}.
$$

(iv) Let
$D=\{P_1,P_2,P_3,P_4,P_5,P_6\}=\{(0,0),(0,1),(1,0),(2,1),(3,0),(2,2)\}$.
There are 3 blocks in $D$, namely $\{P_1\}$, $\{P_2,P_3\}$ and
$\{P_4,P_5,P_6\}$. Then
$$\varphi(D)=\varphi(\{(0,0)\})\cdot\varphi(\{(-1,1),(0,0)\})\cdot \varphi(\{(-1,1),(0,0),(-1,2)\}), $$
i.e.

$$
\varphi\big{(}\hspace{16pt} \begin{picture}(37,25)
\put(-3,-3){$\bullet$}\put(-3,7){$\bullet$}\put(7,-3){$\bullet$}\put(17,7){$\bullet$}
\put(27,-3){$\bullet$}\put(17,17){$\bullet$}
\boxs{-10,0}\boxs{0,0}\boxs{10,0}
\boxs{-10,10}\boxs{0,10}\boxs{10,10}
\linethickness{1pt}\put(0,0){\line(0,1){27}}
\linethickness{1pt}\put(-15,0){\line(1,0){50}}
\end{picture} \big{)}
= 
\varphi\big{(}\hspace{16pt} \begin{picture}(27,25)
\put(-3,-3){$\bullet$}
\boxs{-10,0}\boxs{0,0}\boxs{10,0}
\boxs{-10,10}\boxs{0,10}\boxs{10,10}
\linethickness{1pt}\put(0,0){\line(0,1){27}}
\linethickness{1pt}\put(-15,0){\line(1,0){40}}
\end{picture} \big{)}
\cdot 
\varphi\big{(}\hspace{16pt} \begin{picture}(27,25)
\put(-13,7){$\bullet$}\put(-3,-3){$\bullet$}
\boxs{-10,0}\boxs{0,0}\boxs{10,0}
\boxs{-10,10}\boxs{0,10}\boxs{10,10}
\linethickness{1pt}\put(0,0){\line(0,1){27}}
\linethickness{1pt}\put(-15,0){\line(1,0){40}}
\end{picture} \big{)}
\cdot 
\varphi\big{(}\hspace{16pt} \begin{picture}(27,25)
\put(-13,7){$\bullet$}\put(-3,-3){$\bullet$}\put(-13,17){$\bullet$}
\boxs{-10,0}\boxs{0,0}\boxs{10,0}
\boxs{-10,10}\boxs{0,10}\boxs{10,10}
\linethickness{1pt}\put(0,0){\line(0,1){27}}
\linethickness{1pt}\put(-15,0){\line(1,0){40}}
\end{picture} \big{)}.$$

(v) For $D=\{(-n+1,n-1),(-n+2,n-2),\dots,(-1,1),(0,0)\}$,
$\varphi(D)=\rho_1^{n\choose 2}$.
\end{exmp}

\begin{proof}[Proof of Lemma \ref{varphi_computation_lemma}]
(i) It immediately follows from the condition (\ref{varphi
condition}).

(ii) By definition, $$\varphi(\tilde{D}) =(-1)^{\tilde{k}}
\sum_{\tilde{\sigma}\in S_{m+n}} \text{sgn}(\tilde{\sigma})
\prod_{i=1}^{m+n} \left(\sum \rho_{w_1}\cdots\rho_{w_{b_i}}\right),
$$ where $w_1,\dots,w_{b_i}\in\N$ and $$w_1+...+w_{b_i}=
\tilde{\sigma}(i)-1-a_i-b_i= \left\{\begin{array}{ll}
        \tilde{\sigma}(i)-i, & \textrm{ if  $i\le m$}; \\
        \tilde{\sigma}(i)-1-m-|P_{i-m}|, & \textrm{ if  $i>m$}.
      \end{array}
 \right.$$
If $\tilde{\sigma}(i)<i$ for some $i\le m$, then no
$w_1,\dots,w_{b_i}$ satisfies the condition, $\prod_{i=1}^{m+n}
(\sum \rho_{w_1}\cdots\rho_{w_{b_i}})=0$, hence the summand
corresponding to $\tilde{\sigma}$ does not contribute to
$\varphi(\tilde{D})$. So we only need to consider those
$\tilde{\sigma}$ satisfying $\tilde{\sigma}(i)=i$ ($1\le i\le m$).
Each such $\tilde{\sigma}$ corresponds to a permutation of
$\{m+1,\dots,m+n\}$, and by translation, a permutation of
$\{1,\dots,n\}$. To be precise,
$$\sigma(i-m)=\tilde{\sigma}(i)-m,\quad m+1\le i\le m+n.$$
Then $\tilde{\sigma}(i)-1-m-|P_{i-m}|=\sigma(i-m)-1-|P_{i-m}|$ for
$m+1\le i\le m+n$. Moreover, $$\tilde{k}={{n+m}\choose
2}-\sum_{i=1}^m|Q_i|-\sum_{i=1}^n(|P_i|+m)={n\choose
2}-\sum_{i=1}^n|P_i|=k.$$ Comparing with the definition of
$\varphi(D)$, we conclude that $\varphi(\tilde{D})=\varphi(D)$.

(iii) It suffices to prove the case when $t=1$. Define
$$
\textbf{v}_i=\begin{bmatrix}h(b_1,i-|P_1|)\\h(b_2,i-|P_2|)\\\vdots\\h(b_n,i-|P_n|)\end{bmatrix},
\quad
\textbf{v}'_i=\begin{bmatrix}h(b_1+1,i-|P_1|)\\h(b_2+1,i-|P_2|)\\\vdots\\h(b_n+1,i-|P_n|)\end{bmatrix},
\quad
 0\le i\le n-1.$$
By the definition of the map $\varphi$,
$$\varphi(D)=(-1)^k\det(\textbf{v}_0,\dots,\textbf{v}_{n-1}),
\quad
\varphi(\tilde{D})=(-1)^k\det(\textbf{v}'_0,\dots,\textbf{v}'_{n-1}).
$$
By the definition of the function $h$, it is easy to deduce the
relation
$$h(b+1,\w)=h(b,\w)+\rho_1h(b,\w-1)+\rho_2h(b,\w-2)+\cdots.$$ Since $|P_1|,\dots,|P_n|$ are non-negative integers,
the above relation implies
$$\textbf{v}_i'=\textbf{v}_i+\rho_1 \textbf{v}_{i-1}+\rho_2 \textbf{v}_{i-2}+\cdots+\rho_i \textbf{v}_0, \quad 0\le i\le n-1,$$
hence
$$\varphi(D)=(-1)^k\det(\textbf{v}_0,\dots,\textbf{v}_{n-1})=
(-1)^k\det(\textbf{v}'_0,\dots,\textbf{v}'_{n-1})=
\varphi(\tilde{D}).
$$

(iv) Suppose the summand in $\varphi(D)$ corresponding to $\sigma\in
S_n$ does contribute. By the definition of $\varphi(D)$, it is
necessary that $\sigma(j)-1-|P_j|\ge 0$ for $1\le j\le n$. Let $1\le
r\le \ell$. For $j\ge i_r$, we have $|P_j|\ge |P_{i_r}|$, therefore
$$\sigma(j)\ge 1+|P_j|\ge 1+|P_{i_r}|=i_r.$$
So $\sigma$ maps the set $\{i_r,i_r+1,\dots,n\}$ to itself for every
$r$. It follows that $\sigma$ maps each block to itself. Let
$\sigma_r$ be the restriction of $\sigma$ to
$\{i_r,i_r+1,\dots,i_{r+1}-1\}$. Define $n_r=i_{r+1}-i_r$,
$k_r=\sum_{j=i_r-1}^{i_{r+1}-2}j-\sum_{j=i_r}^{i_{r+1}-1}|P_j|$.
Then by (ii) and a routine computation, we have
$$
\aligned\varphi(D)&=(-1)^{k_1+\cdots+k_\ell}
\sum_{\sigma_1,\dots,\sigma_\ell}
\text{sgn}(\sigma_1)\cdots\text{sgn}(\sigma_\ell)
\prod_{i=1}^{n_1+\cdots+n_\ell} \left(\sum \rho_{w_1}
...\rho_{w_{b_i}}\right)\\
&=\prod_{r=1}^\ell
\bigg{(}(-1)^{k_r}\sum_{\sigma_r}
\text{sgn}(\sigma_r)\prod_{i=1}^{n_r} \left(\sum \rho_{w_1}
...\rho_{w_{b_i}}\right)\bigg{)}\\
&=\prod_{r=1}^\ell \varphi( \{
P_{i_r}-(i_r-1,0),P_{i_r+1}-(i_r-1,0),\dots,P_{i_{r+1}-1}-(i_r-1,0)\}
).
\endaligned$$

(v) We rewrite the definition of $\varphi$ as
\begin{equation}\label{varphi_version2}\varphi(D) =(-1)^{k} \sum_{(\sigma,\{w^{(i)}_j\})}
\left(\text{sgn}(\sigma) \prod_{i=1}^n
\rho_{w^{(i)}_1}\rho_{w^{(i)}_2}\cdots\rho_{w^{(i)}_{b_i}}\right),
\end{equation}
where $\{w^{(i)}_j\}$ is a set of nonnegative integers, $1\le i\le
n$, $1\le j\le b_i$. For $1\le i\le n$, since $|P_i|=0$, those
$w^{(i)}_j$ satisfy the condition
$$w^{(i)}_1+\cdots +w^{(i)}_{b_i}=\sigma(i)-1.$$
Denote by $\Sigma$ the set of all possible data $(\sigma,
\{w^{(i)}_j\})$. Let $\Sigma'\subset\Sigma$ be the subset consisting
of those
 $(\sigma, \{w^{(i)}_j\})$ such that not all $w^{(i)}_j$ are $0$ or $1$.
 We shall define a `conjugation' on the set $\Sigma'$, i.e. an automorphism $f:\Sigma'\to\Sigma'$
such that $f\circ f$ is the identity.

For $(\sigma, \{w^{(i)}_j\})\in\Sigma'$, define $m_i$ to be the
number of nonzero elements in $(w^{i}_1,\dots,w^{(i)}_{b_i})$, for
$1\le i\le n$. Then
$$m_1+\cdots+m_n\le 0+1+\cdots+(n-1)={n\choose 2}.$$
Since some $w^{(i)}_j$ is greater than $1$, the inequality must be
strict, therefore we can find a smallest pair $(r,r')$ such that
$r<r'$ and $m_r=m_{r'}$. (Here we use the lexicographic order, i.e.,
$(r,r')<(s,s')$ if $r<s$ or ($r=s$ and $r'<s'$).) Let
$$\{j_1<\cdots<j_{m_r}\}:=\{j| \; w^{(r)}_j\neq 0\},$$
$$\{j'_1<\cdots<j'_{m_r}\}:=\{j'| \; w^{(r')}_j\neq 0\}.$$
Define $\tilde{\sigma}=\sigma \cdot(r, r')$, i.e.
$\tilde{\sigma}(r)=\sigma(r'), \tilde{\sigma}(r')=\sigma(r)$,
$\tilde{\sigma}(\ell)=\sigma(\ell)$ for $\ell\neq r,r'$. Define
$\{\tilde{w}^{(i)}_j\}$ as follows. For $i\neq r,r'$, define
$\tilde{w}^{(i)}_j=w^{(i)}_j$ for $1\le j\le b_i$. For $i=r$, define
$$\tilde{w}^{(r)}_{j_\ell}=w^{(r')}_{j'_\ell} \hbox{ for }1\le\ell\le m_r,\quad \hbox{ and } \tilde{w}^{(r)}_j=0 \hbox{ for } j\neq j_1,\dots,j_{m_r}.$$
Similarly for $i=r'$, define
$$\tilde{w}^{(r')}_{j'_\ell}=w^{(r)}_{j_\ell} \hbox{ for }1\le\ell\le m_r,\quad \hbox{ and } \tilde{w}^{(r')}_{j'}=0 \hbox{ for } j'\neq j'_1,\dots,j'_{m_r}.$$
Define the conjugation $f:(\sigma,\{w^{(i)}_j\})\mapsto
(\tilde{\sigma},\{\tilde{w}^{(i)}_j\})$. It is immediate from the
above construction that $f$ is a conjugation. Moreover, $f$ has no
fixed point because $\sigma\neq\tilde{\sigma}$. Since
$\text{sgn}(\sigma)=-\text{sgn}(\tilde{\sigma})$, the summand in
(\ref{varphi_version2}) corresponding to $(\sigma,\{w^{(i)}_j\})$
cancels with the summand corresponding to
$(\tilde{\sigma},\{\tilde{w}^{(i)}_j\})$.

Finally, we are left with the case when all $w^{(i)}_j$ are $0$ or
$1$. Using Definition \ref{linearcompar} (b), and using the fact
that the monomial $\rho_1^w$ in $h(b,\w)$ has coefficient
${b\choose\w}$, we obtain

$$\varphi(D)=(-1)^{n\choose 2}\begin{vmatrix}
{b_1\choose 0}\rho_1^{0}&     {b_1\choose 1}\rho_1^{1}&\cdots& {b_1\choose n-1}\rho_1^{n-1}\\
\vdots&\vdots&\ddots&\vdots\\
{b_n\choose 0}\rho_1^{0}&     {b_n\choose 1}\rho_1^{1}&\cdots& {b_n\choose n-1}\rho_1^{n-1}\\
\end{vmatrix}=\begin{vmatrix}
{b_n\choose 0}&     {b_n\choose 1}&\cdots& {b_n\choose n-1}\\
\vdots&\vdots&\ddots&\vdots\\
{b_1\choose 0}&     {b_1\choose 1}&\cdots& {b_1\choose n-1}\\
\end{vmatrix}\rho_1^{n\choose 2}=c\cdot\rho_1^{n\choose 2},$$
where $c$ is the second determinant. Notice that ${b\choose
i}=b(b-1)\cdots(b-i+1)/i!$ is a polynomial of $b$ of degree $i$
whose leading term is $b^i/i!$. By appropriate column operations,
i.e., adding appropriate multiples of the first $i-1$ columns to the
$i$-th column for $1\le i\le n$, we obtain
$$c=\begin{vmatrix}
{b_n\choose 0}&     {b_n\choose 1}&\cdots& {b_n\choose n-1}\\
\vdots&\vdots&\ddots&\vdots\\
{b_1\choose 0}&     {b_1\choose 1}&\cdots& {b_1\choose n-1}\\
\end{vmatrix}
=\begin{vmatrix}
1& b_n&\frac{b_n^2}{2!}&\cdots& \frac{b_n^{n-1}}{(n-1)!}\\
\vdots&\vdots&\vdots&\ddots&\vdots\\
1& b_1&\frac{b_1^2}{2!}&\cdots& \frac{b_1^{n-1}}{(n-1)!}\\
\end{vmatrix}
=\frac{\prod_{i<j}(b_i-b_j)}{1!2!\cdots(n-1)!},
$$
by using standard results of Vandermonde matrices. Since
$b_1>b_2>\cdots>b_n$ are distinct integers  by assumption, $c$ is a
strictly positive integer.

(vi) Follows immediately from Definition \ref{linearcompar} (b).
\end{proof}

\subsection{Relation between $\varphi(D)$ and $\Delta(D)$}

We need the following elementary lemma.
\begin{lem}[\cite{LL}, Lemma 26]\label{sum} For $(\alpha_i,\beta_i)\in \N\times\N$ ($1\le
i\le n$) and $c,e\in\N$,
$$
(\sum_{i=1}^nx_i^c y_i^e)\cdot\begin{vmatrix}
                                      x_{1}^{\alpha_{1}} y_{1}^{\beta_{1}} & x_{1}^{\alpha_{2}} y_{1}^{\beta_{2}} &\cdots & x_{1}^{\alpha_{n}} y_{1}^{\beta_{n}} \\
                                      x_{2}^{\alpha_{1}} y_{2}^{\beta_{1}} & x_{2}^{\alpha_{2}} y_{2}^{\beta_{2}} &\cdots & x_{2}^{\alpha_{n}} y_{2}^{\beta_{n}} \\
                                      \vdots & \vdots & \ddots &\vdots \\
                                      x_{n}^{\alpha_{1}} y_{n}^{\beta_{1}} & x_{n}^{\alpha_{2}} y_{n}^{\beta_{2}} &\cdots & x_{n}^{\alpha_{n}} y_{n}^{\beta_{n}} \\
                                      \end{vmatrix}
=\sum_{i=1}^n\begin{vmatrix}
                                      x_{1}^{\alpha_{1}} y_{1}^{\beta_{1}} &\cdots & x_{1}^{\alpha_{i}+c} y_{1}^{\beta_{i}+e} &\cdots & x_{1}^{\alpha_{n}} y_{1}^{\beta_{n}} \\
                                      x_{2}^{\alpha_{1}} y_{2}^{\beta_{1}} &\cdots & x_{2}^{\alpha_{i}+c} y_{2}^{\beta_{i}+e} &\cdots & x_{2}^{\alpha_{n}} y_{2}^{\beta_{n}} \\
                                      \vdots & \ddots &\vdots & \ddots &\vdots \\
                                      x_{n}^{\alpha_{1}} y_{n}^{\beta_{1}} &\cdots &x_{n}^{\alpha_{i}+c} y_{n}^{\beta_{i}+e} &\cdots & x_{n}^{\alpha_{n}} y_{n}^{\beta_{n}} \\
                                      \end{vmatrix}.
$$
As a consequence,
$$0\equiv\sum_{i=1}^n\;\Delta(\{(\alpha_1,\beta_1),\dots,(\alpha_{i-1},\beta_{i-1}),(\alpha_i+c,\beta_i+e),(\alpha_{i+1},\beta_{i+1}),\dots,(\alpha_n,\beta_n)\})$$
modulo lower degrees.
\end{lem}

Let us recall the definition of minimal staircase forms defined in
\cite{LL}, and then define special minimal staircase forms.
\begin{defn} We call $D=\{P_1,\dots,P_n\}\in\D_n$ a \emph{minimal staircase form} if
$|P_i|=i-1$ or $i-2$ for every $1\le i\le n$. For a minimal
staircase form $D$, let $\{i_1<i_2<\dots<i_\ell\}$ be the set of
$i$'s such that $|P_i|=i-1$, we define the \emph{partition type} of
$D$ to be the partition of (${n\choose 2}-\sum|P_i|$) consisting of
all the positive integers in the sequence
$$(i_1-1,i_2-i_1-1,i_3-i_2-1,\dots,i_\ell-i_{\ell-1}-1, n-i_\ell).$$
\end{defn}
\begin{exmp}
  For $n=8$ and $D=\{P_1,\dots,P_8\}$ satisfying
  $(|P_1|,\dots,|P_8|)=(0,1,1,2,4,4,5,6)$, the set
  $\{i\,\big{|}\, |P_i|=i-1\}$ equals $\{1,2,5\}$. The positive integers in the sequence $(1-1,2-1-1,5-2-1,8-5)$ is $(2,3)$, so the partition type of $D$ is
  $(2,3)$.
\end{exmp}

\begin{defn}\label{df:special minimal staircase form}
The data
$(m,n,(r_1,\dots,r_m),(s_1,\dots,s_m))\in\N\times(\N^+)\times\N^m\times\N^m$
satisfying $1\le r_1<r_2<\dots<r_m<r_{m+1}:=n$ and $0\le s_i\le
r_{i+1}-r_i-1$ ($1\le i\le m$) determines a $D\in\D_n$ as follows.
$$\aligned D=\{(0,0),(1,0),\cdots,(n-1,0)\}&\cup\{(r_1-1,1),(r_2-1,1),\dots,(r_m-1,1)\}
\\&\setminus\{(r_1+s_1,0),(r_2+s_2,0),\dots,(r_m+s_m,0)\}.\endaligned$$
We call $D$ a \emph{special minimal staircase form}.
\end{defn}
\begin{rem}
It is easy to see that a special minimal staircase form is indeed a
minimal staircase form. Using the notation in the definition, the
partition type of a special minimal staircase form $D$ is obtained
from $(s_1,s_2,\dots,s_m)$ by eliminating 0's and sorting the
sequence if necessary. The following picture gives a typical example
of a special minimal staircase form,
    $$
    
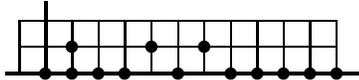
\begin{picture}(150,25)
    \put(-3,-3){$\bullet$}\put(7,-3){$\bullet$}\put(7,7){$\bullet$}\put(17,-3){$\bullet$}
    \put(27,-3){$\bullet$}
    \put(37,7){$\bullet$}\put(47,-3){$\bullet$}
    \put(57,7){$\bullet$}\put(77,-3){$\bullet$}\put(87,-3){$\bullet$}\put(97,-3){$\bullet$}\put(107,-3){$\bullet$}
    \put(67,-3){$\bullet$}
    \boxs{-10,0}\boxs{0,0}\boxs{10,0}\boxs{20,0}\boxs{30,0}\boxs{40,0}
    \boxs{50,0}\boxs{60,0}\boxs{70,0} \boxs{80,0}\boxs{90,0}\boxs{100,0}
    \boxs{-10,10}\boxs{0,10}\boxs{10,10}\boxs{20,10}\boxs{30,10}\boxs{40,10}
    \boxs{50,10}\boxs{60,10}\boxs{70,10}\boxs{80,10}\boxs{90,10}\boxs{100,10}
    \linethickness{1pt}\put(0,0){\line(0,1){27}}
    \linethickness{1pt}\put(-15,0){\line(1,0){135}}
    \end{picture}
    $$
where $m=3$, $n=13$, $(r_1,r_2,r_3)=(2,5,7)$,
$(s_1,s_2,s_3)=(2,1,5)$, and the partition type is $(1,2,5)$.
\end{rem}

Let us recall the following two facts proved in \cite{LL}.
\begin{lem}[Minors Permuting Lemma in \cite{LL}]\label{lem:Minors Permuting Lemma} Let $D=\{P_1,\dots,P_n\}\in\mathfrak{D}$, $h, \ell, m\in\N^+$ satisfy $2\le h<h+\ell+m\le n+1$,
$|P_h|=h-1,|P_{h+\ell}|=h+\ell-1$, $|P_{h+\ell+m}|=h+\ell+m-1$ (this
condition holds if $h+\ell+m=n+1$ by assumption). Suppose
$|P_{h+\ell}|_x,...,|P_{h+\ell+m-1}|_x\geq \ell$. Define
$$\aligned D'=\{&P_1,P_2, \dots, P_{h-1},
P_{h+\ell}-(\ell,0),P_{h+\ell+1}-(\ell,0),\dots,P_{h+\ell+m-1}-(\ell,0),\\
&P_{h}+(m,0),P_{h+1}+(m,0),\dots,P_{h+\ell-1}+(m,0),
P_{h+\ell+m},\dots,P_{n}\}.\endaligned$$ Then
$\Delta(D)\equiv\Delta(D')$ {\rm(modulo lower degrees)}.
\end{lem}

\begin{lem}[Main Theorem of \cite{LL}]\label{thm:Main Theorem of LL}
Suppose $k$ is a positive integer such that $n\ge 8k+5$ and  $d_1,
d_2\ge (2k+1)n$ are two integers whose sum is $n(n-1)/2-k$. Then
$M_{d_1,d_2}$ is minimally generated by $p(k)$ elements, i.e., $\dim
M_{d_1,d_2}=p(k)$. Furthermore, there is a one-to-one correspondence
between partitions of $k$ and generators, namely
$$
(\mu=\sum m_i j_i\in\Pi_k)\longleftrightarrow (\emph{a minimal
staircase form of partition type } \mu).
$$
\end{lem}

The following lemma is essential for this paper.
\begin{lem}\label{lem:linear independence}
{\rm(i)} Let $d_1,d_2,k\in \N$ and $d_1+d_2={n\choose 2}-k$. Define
$$\Pi'_k=\left\{\mu\in\Pi_k
\left|\begin{array}{l} \textrm{there exists a minimal staircase form
$F_\mu\in\D_n$ of}\\
\textrm{partition type $\mu$ and of bi-degree $(d_1,d_2)$}\\
\end{array} \right.\right\}.
$$
If there are coefficients $a_\mu\in \C$ ($\mu\in\Pi'_k$) satisfying
$$\sum_{\mu\in\Pi'_k} a_\mu\, \Delta(F_\mu)\equiv 0\quad \textrm{   \rm (modulo
lower degrees)},$$ then $a_\mu=0$, $\forall\mu\in\Pi'_k$. In other
words, $\{\Delta(F_\mu)\}_{\mu\in\Pi'_k}$ form a linearly
independent set in $M_{d_1,d_2}$.

{\rm(ii)} Any two special minimal staircase form in $\D_n$ of the
same partition type and the same bi-degree are equivalent modulo
lower degrees.
\end{lem}
\begin{proof}
{(i)} Choose a sufficiently large integer $N$ and choose $(N-n)$
points $P_{n+1},\dots,P_N\in\N\times\N$ such that $|P_i|=i-1$ for
$n+1\le i\le N$ and
$$|P_{n+1}|_x+\cdots+|P_N|_x\ge (2k+1)N,\quad |P_{n+1}|_y+\cdots+|P_N|_y\ge (2k+1)N.$$ let $F'_\mu=F_\mu\cup\{P_{n+1},P_{n+2},\dots,P_N\}$.
The condition of Lemma \ref{thm:Main Theorem of LL} is satisfied, so
$\Delta(F'_\mu)$ for $\mu\in\Pi'_k$ are linearly independent modulo
lower degrees. But $\Delta(F'_\mu)$ is equivalent to
$\Delta(F_\mu)\cdot f_0$ for a polynomial $f_0$ independent of
$\mu$. (In fact $f_0=\prod_{i,j}a_{ij}$ for $1\le i<j\le N$ and
$j\ge n+1$, with appropriate choices $a_{ij}=x_j-x_i$ or $y_j-y_i$.
We do not need the exact formula here.) So the linear independence
of $\{\Delta(F'_\mu)\}_{\mu\in\Pi'_k}$ implies the linear
independence of $\{\Delta(F_\mu)\}_{\mu\in\Pi'_k}$.

(ii) The claim follows immediately from Minors Permuting Lemma
(Lemma \ref{lem:Minors Permuting Lemma}).
\end{proof}

\begin{prop}\label{prop:varphi and Delta} Let $n\in\N^+$, $D=\{P_1,\dots,P_n\}\in \D$ and $k={n\choose2}-\sum_{i=1}^n|P_i|\ge 0$.
Suppose $N\in\N^+$ satisfies $N>N_0:=(\sum_{i=1}^n|P_i|_y)(k+1)$.
Define
$$\tilde{D}=\{(0,0),(1,0),\dots,(N-1,0),P_1+(N,0),\dots,P_n+(N,0)\}\in\D_{N+n}.$$
Let $d_2=\sum_i |P_i|_y$ be the $y$-degree of $D$ (which is also the $y$-degree of
$\tilde{D}$). For $\mu\in\Pi_{d_2,k}$, let $F_\mu$ be a special
minimal staircase form with the same bi-degree as $\tilde{D}$ and be
of partition type $\mu$. Then there exist unique integers $a_\mu$
($\mu\in\Pi_{d_2,k}$) such that
$$
\Delta(\tilde{D}) \equiv \sum_{\mu \in \Pi_{d_2,k}} a_\mu\cdot
\Delta(F_\mu)\quad \emph{(modulo lower degrees)}.$$ In fact, the
integers $a_\mu$ satisfy
\begin{equation}\label{eq:formula for amu}\varphi(D)=\sum_{\mu\in\Pi_{d_2,k}} a_\mu \rho_\mu.
\end{equation}

\end{prop}

\begin{proof} In this proof we use $D\in\D$ that does
not satisfy the assumption (\ref{order}).

The uniqueness of $a_\mu$ follows from the fact that $\{\Delta(F_\mu)\}_{\mu\in\Pi'_k}$ form a linearly
independent set in $M_{d_1,d_2}$, proved in Lemma \ref{lem:linear independence}. For the existence of $a_\mu$,
we shall give an algorithm showing that those $a_\mu$ are exactly the integers satisfying (\ref{eq:formula for amu}).

We separate the set $\{(1,0),(2,0),\dots,(N_0,0)\}$ into
$(\sum_{i=1}^n|P_i|_y)$ segments, where the $r$-th segment for $1\le
r\le (\sum_{i=1}^n|P_i|_y)$ consists of $(k+1)$ points $\{(i,0)|\;
(r-1)(k+1)+1\le i\le r(k+1)\}$.

Consider the following sequence of length $d_2$.
$$(\star):\quad (1,|P_1|_y),\dots,(1,2),(1,1),\; (2,|P_2|_y),\dots,(2,2), (2,1),\; \dots,\; (n,|P_n|_y),\dots,(n,2),(n,1)$$
with the natural total order that $(i',j')<(i,j)$ if $(i',j')$ is to
the left of $(i,j)$. For $(i,j)$ in the above sequence, define
$r(i,j)\in\N^+$ to be the integer such that$(i,j)$ is the
$r(i,j)$-th pair in the sequence $(\star)$.

Denote $Q^{(0)}_s=(s-1,0)$ for $1\le s\le N$ and
$P^{(0)}_t=P_t+(N,0)$ for $1\le t\le n$ and denote
$$D^{(0)}:=\tilde{D}=\{Q^{(0)}_1,Q^{(0)}_2,\dots,Q^{(0)}_N,P^{(0)}_1,P^{(0)}_2,\dots,P^{(0)}_n\}.$$
Given a set of nonnegative integers
$\textbf{w}=\{w^{(i')}_{j'}\}_{(i',j')\in (\star)}$, we construct
$$D^{(r)}=\{Q^{(r)}_1\dots,Q^{(r)}_N,P^{(r)}_1,\dots,P^{(r)}_n\}$$
inductively on $r\in[1,\sum_{\ell=1}^n|P_\ell|_y]$. Here we do not
require $D^{(r)}$ to satisfy the assumption of order defined in
(\ref{order}). Suppose $D^{(r-1)}$ has been constructed and the
$r$-th element in the sequence $(\star)$ is the pair $(i,j)$. Then
$D^{(r)}$ is constructed as follows.
$$\aligned
&Q^{(r)}_{(r-1)(k+1)+2+w^{(i)}_j}=((r-1)(k+1),1) ;\\
&Q^{(r)}_\ell=Q^{(r-1)}_\ell,\quad \hbox{ for } 1\le\ell\le N\hbox{ and } \ell\neq (r-1)(k+1)+2+w^{(i)}_j;\\
&P^{(r)}_{i}=P^{(r-1)}_{i}+(w^{(i)}_j+1,-1) ;\\
&P^{(r)}_\ell=P^{(r-1)}_\ell,\quad \hbox{ for } 1\le\ell\le n\hbox{ and } \ell\neq i.\\
\endaligned$$

The following can be proved inductively on $r$:
\begin{equation}  \label{eq:delta}
\Delta(\tilde{D})\equiv (-1)^r\sum_\textbf{w}
\Delta(D^{(r)}_{\textbf{w}}),\quad \hbox{ (modulo lower degrees)}
\end{equation}
where $\textbf{w}$ runs through all possible sets of integers
$\{w^{(i')}_{j'}\}_{(i',j')\le (i,j)}$ where
$w^{(i')}_{j'}\in[0,k]$. Indeed, for $r=1$ and $|P_1|_y>0$ (the case
$|P_1|_y=0$ is similar) we need to show that (for simplicity of
notation we use $w$ in place of $w^{(1)}_{|P_1|_y}$)
$$\Delta(\tilde{D})+\sum_{0\le w\le k} \Delta(\{(0,0),\dots,(w,0),(0,1),(w+2,0),\dots,(N-1,0), P_1+(w+1,-1),P_2,\dots,P_n\})$$
is equivalent to 0 modulo lower degrees. But this is an immediate
consequence of Lemma \ref{sum} by plugging in
$(c,e)=P^{(0)}_1-(0,1)$ and
$$\{(\alpha_1,\beta_1),(\alpha_2,\beta_2),\dots\}=\{(0,0),(1,0),\dots,(N-1,0),(0,1),P^{(0)}_2,P^{(0)}_3,\dots,P^{(0)}_n\}.$$
Here we can assume $w\le k$ because otherwise the total degree of
the polynomial $\Delta(\tilde{D})$ is strictly greater than
$N+n\choose 2$ hence
 $\Delta(\tilde{D})\equiv 0$ modulo lower degrees.

 For $r=2$, we only consider the case $|P_1|_y\ge 2$, since the
other case is similar. By induction we have
$$\Delta(\tilde{D})\equiv -\sum_{0\le w^{(1)}_{|P_1|_y}\le k} \Delta(D^{(1)}_{w^{(1)}_{|P_1|_y}}),\quad \hbox{ (modulo lower degrees)}.$$
By a similar argument as in the case $r=1$,
$$\Delta(D^{(1)}_{w^{(1)}_{|P_1|_y}})\equiv -\sum_{0\le w^{(1)}_{|P_1|_y-1}\le k} \Delta(D^{(1)}_{\{w^{(1)}_{|P_1|_y},w^{(1)}_{|P_1|_y-1}\}}),\quad \hbox{ (modulo lower degrees)}.$$
Combine the above two formulas together, we have
$$\Delta(\tilde{D})\equiv (-1)^2\sum_{0\le w^{(1)}_{|P_1|_y},w^{(1)}_{|P_1|_y-1}\le k} \Delta(D^{(1)}_{\{w^{(1)}_{|P_1|_y},w^{(1)}_{|P_1|_y-1}\}}),\quad \hbox{ (modulo lower degrees)}.$$
An easy induction similar to the above argument gives the proof of
(\ref{eq:delta}).

Now look at (\ref{eq:delta}) when $r=r_0=\sum_{\ell=1}^n|P_\ell|_y$.
The $y$-coordinates of $P^{(r_0)}_1,\dots,P^{(r_0)}_n$ are all zero.
A necessary condition for $\Delta(D^{(r_0)}_\textbf{w})\not\equiv0$
is that
$$\{|P^{(r_0)}_1|_x,|P^{(r_0)}_2|_x,\dots,|P^{(r_0)}_n|_x\} \hbox{ is a permutation of }
\{N,N+1,\dots,N+n-1\}$$ and hence we can assume such a condition
holds. Let $\sigma\in S_n$ be the permutation that satisfies
$||P^{(r_0)}_i|_x=\sigma(i)+N-1$. Since
$$P^{(r_0)}_i=P^{(0)}_i+\sum_{j=1}^{|P_i|_y}w^{(i)}_j,$$
we have
$$\sum_{j=1}^{|P_i|_y}w^{(i)}_j=P^{(r_0)}_i-P^{(0)}_i=(\sigma(i)+N-1)-(N+|P_i|)=\sigma(i)-1-|P_i|=\sigma(i)-1-a_i-b_i,$$
which is exactly the condition in the definition of $\varphi(D)$
(cf. Definition \ref{linearcompar}(a)). Next, we shall figure out
the correct sign. For this, we have to rearrange the order of points
in $D^{(r_0)}_\textbf{w}$ to satisfy the condition (\ref{order}).
For $1\le r\le \sum_{\ell=1}^n|P_\ell|_y$, the $r$-th segment
$$\big{(}(r-1)(k+1)+1,0\big{)},\big{(}(r-1)(k+1)+2,0\big{)},\dots,\big{(}(r-1)(k+1)+1+w^{(i)}_j,0\big{)},\dots,\big{(}r(k+1),0\big{)}$$
is modified to
$$\big{(}(r-1)(k+1)+1,0\big{)},\big{(}(r-1)(k+1)+2,0\big{)},\dots,\big{(}(r-1)(k+1),1\big{)},\dots,\big{(}r(k+1),0\big{)}.$$
The only change is that the point
$\big{(}(r-1)(k+1)+1+w^{(i)}_j,0\big{)}$ is replaced by
$\big{(}(r-1)(k+1),1\big{)}$. To rearrange this segment into correct
order, we need to move the $(1+w^{(i)}_j)$-th point in front of the
first point, so the change of sign is $(-1)^{w^{(i)}_j}$. On the
other hand, rearranging $\{P^{(r_0)}_1,\dots,P^{(r_0)}_1\}$ to the
correct order incurs a sign change $\text{sgn}(\sigma)$. So the
overall sign change is
$$(-1)^{\sum_{i=1}^n\sum_{j=1}^{|P_i|_y}w^{(i)}_j}\cdot\text{sgn}(\sigma)
=(-1)^{\sum_{i=1}^n(\sigma(i)-1-|P_i|)}\cdot\text{sgn}(\sigma)=(-1)^k\text{sgn}(\sigma),$$
which coincides with the signs in the definition of $\varphi(D)$
(cf. Definition \ref{linearcompar}(a)).

Finally, note that $D^{(r_0)}_\textbf{w}$ (after rearranging it to
the correct order) is a special minimal staircase form defined in
Definition \ref{df:special minimal staircase form}. The partition
type of $D^{(r_0)}_\textbf{w}$ is $(w^{(i)}_j)_{i,j}$, which is
compatible with the definition (\ref{varphi_version2}) of
$\varphi(D)$. Thus we have finished the proof of Proposition
\ref{prop:varphi and Delta}.
\end{proof}

\section{The upper bound of $\dim M_{d_1,d_2}$}
\subsection{A characterization of the $q,t$-Catalan number.}

Recall the following conjecture we gave in \cite{LL}. We would like to point out that Mahir Can and Nick Loehr
 gave an equivalent conjecture in their unpublished work.
\begin{conj}\label{conj} Let $\Lambda_n$
be the set of integer sequences $\lambda_1\ge\cdots\ge
\lambda_{n-1}\ge\lambda_n= 0$ satisfying $\lambda_i\le n-i$ for all
$i\in[1,n]$. For any
$\lambda=(\lambda_1,\dots,\lambda_n)\in\Lambda_n$, let
$$a_i=n-i-\lambda_i,\quad b_i=\#\{j |\, i<j\le n, \lambda_i-\lambda_j + i-j \in \{0, 1\} \}$$
and define $D(\lambda)=\{(a_i, b_i)|1\le i\le n\}$.
 Then $\{\Delta(D(\lambda))\}_{\lambda\in\Lambda_n}$  generates the ideal $I$.
\end{conj}

\begin{exmp}
  For $n=3$, $\Lambda_3$ consists of $(2,1,0), (1,1,0),
  (2,0,0),(1,0,0),(0,0,0)$, the corresponding $D(\lambda)$ are
$$
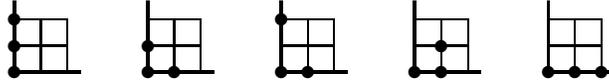
\begin{picture}(27,25)
\put(-3,17){$\bullet$}\put(-3,7){$\bullet$}\put(-3,-3){$\bullet$}
\boxs{0,0}\boxs{10,0}\boxs{0,10}\boxs{10,10}
\linethickness{1pt}\put(0,0){\line(0,1){27}}
\linethickness{1pt}\put(0,0){\line(1,0){25}}
\end{picture}
\quad\quad
\begin{picture}(27,25)
\put(7,-3){$\bullet$}\put(-3,7){$\bullet$}\put(-3,-3){$\bullet$}
\boxs{0,0}\boxs{10,0}\boxs{0,10}\boxs{10,10}
\linethickness{1pt}\put(0,0){\line(0,1){27}}
\linethickness{1pt}\put(0,0){\line(1,0){25}}
\end{picture}
\quad\quad
\begin{picture}(27,25)
\put(-3,17){$\bullet$}\put(7,-3){$\bullet$}\put(-3,-3){$\bullet$}
\boxs{0,0}\boxs{10,0}\boxs{0,10}\boxs{10,10}
\linethickness{1pt}\put(0,0){\line(0,1){27}}
\linethickness{1pt}\put(0,0){\line(1,0){25}}
\end{picture}
\quad\quad
\begin{picture}(27,25)
\put(7,7){$\bullet$}\put(7,-3){$\bullet$}\put(-3,-3){$\bullet$}
\boxs{0,0}\boxs{10,0}\boxs{0,10}\boxs{10,10}
\linethickness{1pt}\put(0,0){\line(0,1){27}}
\linethickness{1pt}\put(0,0){\line(1,0){25}}
\end{picture}
\quad\quad
\begin{picture}(27,25)
\put(17,-3){$\bullet$}\put(7,-3){$\bullet$}\put(-3,-3){$\bullet$}
\boxs{0,0}\boxs{10,0}\boxs{0,10}\boxs{10,10}
\linethickness{1pt}\put(0,0){\line(0,1){27}}
\linethickness{1pt}\put(0,0){\line(1,0){25}}
\end{picture}
$$
\end{exmp}

We shall not prove the conjecture in this paper. Instead, we give a
characterization of $D(\lambda)$ appeared in the conjecture. This
characterization will be used to provide an upper bound of $\dim
M_{d_1,d_2}$.

\begin{defn}\label{defn:Dcatalan}
Let $\D_n^{catalan}$ be the set consisting of $D\subset\N\times\N$,
where $D$ contains $n$ points  satisfying the following conditions.

(a) If $(p,0)\in D$ then $(i,0)\in D, \forall i\in [0,p]$.

(b) For any $p\in\N$,
$$\#\{j\,|\, (p+1,j)\in D\}+\#\{j\,|\,
(p,j)\in D\}\ge\max\{j\,|\, (p,j)\in D\}+1.$$ (If $\{j\,|\, (p,j)\in
D\}=\emptyset$, then we require that no point $(i,j)\in D$ satisfies
$i\ge p$.)
\end{defn}
\medskip

\begin{prop}\label{prop:theta} The map $\theta: \Lambda_n\to\D_n^{catalan}$ sending $\lambda$ to $D(\lambda)$
is one-to-one.
\end{prop}
\begin{proof}
  We first show that $D(\lambda)$ is in $\D_n^{catalan}$, i.e., it
  satisfies conditions (a)(b) of Definition \ref{defn:Dcatalan}.

  By the definition of $D(\lambda)$, suppose $a_i=a_{i'}$ for some $i,i'\in[1,n]$, then
  $i\le i'$ if and only if $b_i\ge b_{i'}$. Indeed, suppose $i\le
  i'$. Since $a_i=a_{i'}$ implies
  $(\lambda_i+i)=(\lambda_{i'}+i')$, we have
  $$\{j |\, i<j\le n, (\lambda_i+i)-(\lambda_j + j) \in \{0, 1\}
  \}\supseteq \{j |\, i'<j\le n, (\lambda_{i'}+i')-(\lambda_j+j) \in \{0, 1\}
  \},$$
  hence $b_i\ge b_{i'}$.

  For (a), suppose $(a_\ell,b_\ell)=(p,0)\in D(\lambda)$ and $(p-1,0)\notin D(\lambda)$. Since
  $$a_i-a_{i+1}=(n-i-\lambda_i)-(n-i-1-\lambda_{i+1})=1-(\lambda_i-\lambda_{i+1})\le 1, \quad \forall i\in [1,n-1]$$
  and $a_n=0$, there exists $i\in [\ell+1,n]$ such that $a_i=p-1$.
  Suppose $i_0$ is maximal among all such $i$. Since $(a_\ell,b_\ell)=(p,0)$, we have $a_i<p$ for all $i>\ell$.
  Therefore
  $$b_{i_0}=\#\{j |\, i_0<j\le n,\, a_j-a_{i_0} \in \{0, 1\}
  \}=\#\{j |\, i_0<j\le n,\, a_j\in \{p-1, p\}
  \}=0,$$
  $(a_{i_0},b_{i_0})=(p-1,0)$,
  which contradicts our assumption that $(p-1,0)\notin D(\lambda)$.

  For (b), if $\{j\,|\, (p,j)\in
D\}=\emptyset$, then since $a_i-a_{i+1}\le 1\, \forall i$, there is
no point in $D$ whose $x$-coordinate is greater than or equal to
$p$. Now assume $\{j\,|\, (p,j)\in D\}\neq\emptyset$, define
$q=\max\{j\,|\, (p,j)\in D\}$, and $(a_\ell,b_\ell)=(p,q)\in D$. By
the definition of $b_\ell$,
$$q=b_\ell=\#\{j|\, \ell<j\le n, \, a_j-a_\ell\in\{0,1\}\}=\#\{j|\, \ell<j\le n, \, a_j=p \hbox{ or } p+1\},$$
therefore,
$$\#\{j\,|\, (p+1,j)\in D\}+\#\{j\,|\,
(p,j)\in D\}\ge q+1.$$ So $D(\lambda)$ is in $\D_n^{catalan}$.

To show that $\theta:D\mapsto D(\lambda)$ is a bijection, it
suffices to construct a map $\theta^{-1}$ sending $D(\lambda)$ back
to $\lambda$. We give an inductive construction on $n$. Let $p\in\N$
be the minimal integer such that
$$\#\{j\,|\, (p+1,j)\in D\}+\#\{j\,|\,(p,j)\in D\}\le \max\{j\,|\, (p,j)\in D\}+1.$$
Let $q=\max\{j\,|\, (p,j)\in D\}$ and define $(a_1,b_1)=(p,q)\in D$.
Now $D':=D\setminus \{(a_1,b_1)\}$ has $(n-1)$ points and we can
check that it is in $\D_{n-1}^{catalan}$. By induction we have
$\theta^{-1}(D')=(\lambda'_1,\lambda'_2,\cdots,\lambda'_{n-1})$.
Then we define
$$\theta^{-1}(D)=(n-1-p,\lambda'_1,\lambda'_2,\dots,\lambda'_{n-1}).$$
To check that it is in $\Lambda_n$, we need to show
$n-1-p\ge\lambda'_1$, i.e., $p\le (n-1)-\lambda'_1=a'_1+1$, where
$a'_1$ is the minimal integer that
$$\#\{j\,|\, (a'_1+1,j)\in D'\}+\#\{j\,|\,(a'_1,j)\in D'\}\le \max\{j\,|\, (a'_1,j)\in D'\}+1.$$
But $D$ and $D'$ coincide on column $0,1,\dots,p-1$, therefore
$a'_1\ge p-1$.

To check that $\theta $ and $\theta^{-1}$ are inverse to each other
is routine and we shall skip.
\end{proof}

\begin{remark}
The above proposition is also discovered independently by
Alexander Woo \cite{W}.
\end{remark}

\begin{cor}\label{cor:Dcatalan}
  The dimension of $M_{d_1,d_2}$, i.e., the coefficient of $q^{d_1}t^{d_2}$ in the $q,t$-Catalan
  number $C_n(q,t)$, is equal to the number of $D\in\D_n^{catalan}$
  such that the $x$-degree (resp.  $y$-degree) of
  $D$ is $d_1$ (resp. $d_2$).
\end{cor}
\begin{proof}
  It is an immediate consequence of Proposition \ref{prop:theta}, by
  using Garsia and Haglund's description of $q,t$-Catalan number (\cite{GH:pos},
  \cite{GH:proof}), which asserts, in notations of Conjecture
  \ref{conj}, that
$$C_n(q,t)=\sum_{\lambda\in\Lambda}q^{\sum a_i}t^{\sum b_i}.$$
\end{proof}

\subsection{The upper bound of $\dim M_{d_1,d_2}$.}
In order to compare $M_{d_1,d_2}$ for different $n$, we use $M^{(n)}_{d_1,d_2}$ to specify which $n$ we are considering.
\begin{prop}\label{prop:upper bound} Let $\ell,n\in\N^+$.
Then we have
  $$\dim M^{(n)}_{d_1,d_2}\le \dim M^{(n+\ell)}_{d_1+{\ell\choose 2}+n\ell,\, d_2}.$$
In particular, let $k={n\choose 2}-d_1-d_2$, then  $$\dim
M_{d_1,d_2}\le p(d_2,k).$$
\end{prop}
\begin{proof}
For any $D^{(n)}\in\D_n^{catalan}$ whose bi-degree is
$(d_1,d_2)$, we define
$$D^{(n+\ell)}=\{(0,0),(1,0),\dots,(\ell-1,0)\}\cup \big{(}D^{(n)}+(\ell,0)\big{)},$$
where $D^{(n)}+(\ell,0)$ means translating the set $D^{(n)}$ by the vector
$(\ell,0)$. It is easy to verify that $D^{(n+\ell)}\in\D_{n+\ell}^{catalan}$
has bi-degree $(d_1+{\ell\choose 2}+n\ell,d_2)$. By Corollary
\ref{cor:Dcatalan} we have proved the first assertion.

For any $D^{(n)}\in\D_n$ of bi-degree $(d_1,d_2)$, by taking sufficiently large $\ell$ and applying Proposition
\ref{prop:varphi and Delta}, we get
$$
\Delta(D^{(n+\ell)}) \equiv \sum_{\mu \in \Pi_{d_2,k}} a_\mu\cdot
\Delta(F_\mu)\quad \textrm{\rm (modulo lower degrees)},$$ where
$F_\mu\in \D_{n+\ell}$ are special minimal staircase forms of bi-degree
$(d_1+{\ell\choose 2}+n\ell,d_2)$ and of partition type $\mu$. This implies
$$\dim M^{(n+\ell)}_{d_1+{\ell\choose 2}+n\ell,\, d_2}\le p(d_2,k),$$ therefore
$$\dim M^{(n)}_{d_1,\, d_2}\le\dim M^{(n+\ell)}_{d_1+{\ell\choose 2}+n\ell,\, d_2}\le p(d_2,k).$$
\end{proof}

\section{The lower bound of $\dim M_{d_1,d_2}$}

\subsection{A homogeneous term order, leading terms and leading monomials.}\label{total_order3}

\begin{defn}\label{total_order_par_k=n-3} (a) Let $k\in \N^+$ and denote by $\Q[\rho]_k\subset\Q[\rho]$ the vector space spanned by monomials
$\rho_\nu=\prod\rho_{\nu_i}$ for all sequences of positive integers
$\nu=(\nu_1\leq \nu_2 \leq....\leq\nu_m)$ satisfying $\sum\nu_i=k$.
Equivalently, $\Q[\rho]_k$ is the set of weighted homogeneous
polynomials of weight $k$ by assigning the weight of $\rho_i$ to be
$i$, $\forall i\in\N^+$.

(b) For any $\nu,\mu\in\Q[\rho]_k$, denoted by $\nu=(\nu_1\leq \nu_2
\leq....\leq\nu_m)$ and $\mu=(\mu_1\leq \mu_2 \leq....\leq \mu_n)$,
we define $\rho_\nu<\rho_\mu$ if there is a positive integer
$j\le\min(m,n)$ such that $\nu_i=\mu_i \text{ for }1\leq i\leq j-1$
and ${\nu}_j<{\mu}_j$. This clearly defines a total order on
$\Q[\rho]_k$, $\forall k\in\N^+$.

(c) For $k\in\N^+$ and a nonzero polynomial $f=\sum a_\nu\rho_\nu
\in\Q[\rho]_k$ where $a_\nu\in\Q$, we define the leading monomial of
$f$ to be the
$$\LM(f):=\max\{\rho_\nu| a_\nu\neq 0\},$$
and define the leading term of $f$ to be $\LT(f):=a_\nu\rho_\nu$
where $\rho_\nu=\LM(f)$. For $c\in \Q\setminus\{0\}$, define
$\LT(c)=1$ and $\LM(c)=c$.\qed
\end{defn}
\begin{exmp}
  Let $f=2\rho_1\rho_2\rho_7-5\rho_4\rho_6$. Since $\rho_1\rho_2\rho_7<\rho_4\rho_6$, we have
  $$\LM(f)=\rho_4\rho_6,\quad \LT(f)=-5\rho_4\rho_6.$$
\end{exmp}

\begin{lem}\label{product of lowest terms}
(a) The total order of the monomials in $\Q[\rho]_k$ defined in
Definition \ref{total_order_par_k=n-3} is preserved by
multiplication: let $\rho_\mu,\rho_{\mu'},\rho_\nu$ be monomials in
$\Q[\rho]_k$, then $\rho_\mu\le\rho_{\mu'}$ if and only if
$\rho_\mu\rho_{\nu} \le \rho_{\mu'}\rho_\nu$.

(b) Let $\mu,\nu$ be monomials in $\Q[\rho]_k$ and $\mu',\nu'$ be
monomials in $\Q[\rho]_{k'}$ such that $\rho_\mu\le\rho_{\nu}$ and
$\rho_{\mu'}\le\rho_{\nu'}$. Then $\rho_\mu\rho_{\mu'} \le
\rho_{\nu}\rho_{\nu'}$.

(c) for $f\in\Q[\rho]_k$, $g\in\Q[\rho]_{k'}$ ($k,k'\in\N^+$), we
have $$\LM(fg)=\LM(f)\LM(g),\quad \LT(fg)=\LT(f)\LT(g).$$
\end{lem}

\begin{proof}(a) Denote $\mu=(\mu_1\le\mu_2\le\cdots)$,
$\mu'=(\mu'_1\le\mu'_2\le\cdots)$. To show the ``only if'' part,
suppose $\mu\le\mu'$. It suffices to consider the case when we have
a strict inequality $\mu_1<\mu'_1$. Let $\rho_\xi=\rho_\mu\rho_\nu$
and $\rho_{\xi'}=\rho_{\mu'}\rho_\nu$. Let $\ell$ be the smallest
integer satisfying $\nu_\ell> \mu_1$. Then
$$\aligned &\xi=(\nu_1,\dots,\nu_{\ell-1},\mu_1,\dots),\\
&\xi'=(\nu_1,\dots,\nu_{\ell-1},\min(\mu'_1,\nu_\ell),\dots).\endaligned$$
Since $\mu_1<\min(\mu'_1,\nu_\ell)$, we have $\rho_\xi<\rho_{\xi'}$
by definition and therefore $\rho_\mu\rho_{\nu} <
\rho_{\mu'}\rho_\nu$. On the other hand, the ``if'' part immediately
follows from the ``only if'' part.

(b) Applying (a) twice, we have $\rho_\mu\rho_\nu\ge
\rho_\mu\rho_{\nu'}\ge \rho_{\mu'}\rho_{\nu'}$.

(c) It is an immediate consequence of (b).
\end{proof}

\subsection{The theorems on the lower bound of $\dim M_{d_1,d_2}$}

\begin{thm}\label{main:k<=n-4}
Let $n,k\in\N$, $k\le n-4$. Let $d_1,d_2\in\N^+$, $d_1+d_2={n\choose
2}-k$, $d_2\leq d_1$. Then for each $\nu \in \Pi_{d_2,k}$, there
exists a $D_\nu\in \D_n$, such that $\Delta(D_\nu)$ has bi-degree
$(d_1,d_2)$, and $\LM(\varphi(D_\nu))=\rho_\nu$.
\end{thm}

\begin{thm}\label{main:k<=n-3}
Let $n,k\in\N$, $k\le n-3$. Let $d_1,d_2\in\N^+$, $d_1+d_2={n\choose
2}-k$, $d_2\leq d_1$. Then  for each $\nu \in \Pi_{d_2,k}$, there
exists an alternating polynomial $f_\nu$ of bi-degree $(d_1,d_2)$,
either of the form $\Delta(D)$ or of the form $\Delta(D)-\Delta(D')$
for some $D,D'\in\D_n$, such that $\LM(\varphi(f_\nu))=\rho_\nu$.
Moreover, $\emph{dim } M_{d_1, d_2}=p(d_2, k)$, the partition number
of $k$ into at most $d_2$ parts.
\end{thm}

\begin{remark}
Theorem \ref{main:k<=n-3} gives a positive answer to Conjecture 8 in
\cite{LL}. Theorem \ref{main:k<=n-4} and Theorem \ref{main:k<=n-3}
are proved using the same idea. In the proofs, we give explicit
constructions for $D_\nu$ (in Theorem \ref{main:k<=n-4}) and $f_\nu$
(in Theorem \ref{main:k<=n-3}). The constructions are non-canonical
in the sense that there are choices to make, and it seems that no
choice is more natural than others.
\end{remark}

Before we prove the above two theorems, we shall give an example to
illustrate the idea of the construction.

\begin{exmp}
We illustrate Theorem \ref{main:k<=n-4} by giving a construction of
$D_\nu$ for $n=18$, $k=14$, $(d_1,d_2)=(84,7)$,
$\nu=(1,1,1,2,2,3,4)$. First, we divide $\nu$ into 3 sub-partitions
$\tilde{\nu}_1=(1,1,1)$, $\tilde{\nu}_2=(2,2)$,
$\tilde{\nu}_3=(3,4)$. For each sub-partition $\tilde{\nu}_i$, we
construct $D_i\in\D'$ as follows:
$$%
    \setlength{\unitlength}{1pt}
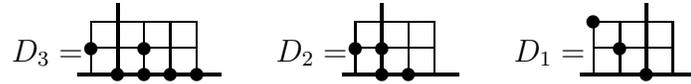
    \begin{picture}(100,30)(-50,0)
    \put(-30,5){$D_3=$}
    \put(0,10){\circle*{5}}\put(10,0){\circle*{5}}\put(20,0){\circle*{5}}
    \put(20,10){\circle*{5}}\put(30,0){\circle*{5}}\put(40,0){\circle*{5}}
    \boxs{0,0}\boxs{10,0}\boxs{20,0}\boxs{30,0}
    \boxs{0,10}\boxs{10,10}\boxs{20,10}\boxs{30,10}
    \linethickness{1pt}\put(10,0){\line(0,1){27}}
    \linethickness{1pt}\put(-5,0){\line(1,0){54}}
    \end{picture}
    \setlength{\unitlength}{1pt}
    \begin{picture}(90,30)(-50,0)
    \put(-30,5){$D_2=$}
    \put(0,10){\circle*{5}}\put(10,0){\circle*{5}}\put(10,10){\circle*{5}}
    \put(20,0){\circle*{5}}
    \boxs{0,0}\boxs{10,0}\boxs{20,0}
    \boxs{0,10}\boxs{10,10}\boxs{20,10}
    \linethickness{1pt}\put(10,0){\line(0,1){27}}
    \linethickness{1pt}\put(-5,0){\line(1,0){44}}
    \end{picture}
    \setlength{\unitlength}{1pt}
    \begin{picture}(90,30)(-50,0)
    \put(-30,5){$D_1=$}
    \put(0,20){\circle*{5}}\put(10,10){\circle*{5}}\put(20,0){\circle*{5}}
    \boxs{0,0}\boxs{10,0}\boxs{20,0}
    \boxs{0,10}\boxs{10,10}\boxs{20,10}
    \linethickness{1pt}\put(20,0){\line(0,1){27}}
    \linethickness{1pt}\put(-5,0){\line(1,0){44}}
    \end{picture}
$$ such that in the term order defined in \S\ref{total_order3}, the leading
monomials $$\LM(\varphi(D_i))=\rho_{\tilde{\nu}_i}, \quad \mbox{ for
} i=1,2,3.$$ Now putting $D_3$, $D_2$, $D_1$ together and adding
appropriate extra points if necessary, we obtain $D_\nu$ as in the
following graph.
    $$\setlength{\unitlength}{1.5pt}
    \begin{picture}(150,50)(0,-5)
    \put(-30,5){$D_\nu=$}
    \put(0,0){\circle*{5}}
    \put(0,10){\circle*{5}}\put(10,0){\circle*{5}}\put(20,0){\circle*{5}}
    \put(20,10){\circle*{5}}\put(30,0){\circle*{5}}\put(40,0){\circle*{5}}
    \put(60,10){\circle*{5}}\put(70,0){\circle*{5}}\put(70,10){\circle*{5}}\put(80,0){\circle*{5}}
    \put(90,20){\circle*{5}}\put(100,10){\circle*{5}}\put(110,0){\circle*{5}}
    \put(140,0){\circle*{5}}\put(150,0){\circle*{5}}\put(160,0){\circle*{5}}\put(170,0){\circle*{5}}
    \boxs{-10,0}\boxs{0,0}\boxs{10,0}\boxs{20,0}\boxs{30,0}\boxs{40,0}
    \boxs{50,0}\boxs{60,0}\boxs{70,0} \boxs{80,0}\boxs{90,0}\boxs{100,0}
    \boxs{110,0}\boxs{120,0}\boxs{130,0}\boxs{140,0}\boxs{150,0}\boxs{160,0}
    \boxs{-10,10}\boxs{0,10}\boxs{10,10}\boxs{20,10}\boxs{30,10}\boxs{40,10}
    \boxs{50,10}\boxs{60,10}\boxs{70,10}\boxs{80,10}\boxs{90,10}\boxs{100,10}
    \boxs{110,10}\boxs{120,10}\boxs{130,10}\boxs{140,10}\boxs{150,10}\boxs{160,10}
    \multiput(-18,23)(2,-2){14}{\circle*{.9}}
    \multiput(24,23)(2,-2){14}{\circle*{.9}}
    \multiput(-18,23)(3,0){14}{\circle*{.9}}
    \multiput(10.5,-5.5)(3,0){15}{\circle*{.9}}
      \put(7,34){\vector(0,-1){10}}
      \put(2,35){{$D_3$}}
    \multiput(42,23)(2,-2){14}{\circle*{.9}}
    \multiput(63,23)(2,-2){14}{\circle*{.9}}
    \multiput(42,23)(3,0){7}{\circle*{.9}}
    \multiput(70.5,-5.5)(3,0){8}{\circle*{.9}}
      \put(52,34){\vector(0,-1){10}}
      \put(47,35){{$D_2$}}
    \multiput(82,23)(2,-2){14}{\circle*{.9}}
    \multiput(94,23)(2,-2){14}{\circle*{.9}}
    \multiput(82,23)(3,0){4}{\circle*{.9}}
    \multiput(110.5,-5.5)(3,0){5}{\circle*{.9}}
      \put(92,34){\vector(0,-1){10}}
      \put(87,35){{$D_1$}}
    \multiput(112,23)(2,-2){14}{\circle*{.9}}
    \multiput(154,23)(2,-2){14}{\circle*{.9}}
    \multiput(112,23)(3,0){14}{\circle*{.9}}
    \multiput(140.5,-5.5)(3,0){14}{\circle*{.9}}
      \put(132,34){\vector(0,-1){10}}
      \put(120,35){{extra points}}
    \linethickness{1pt}\put(0,0){\line(0,1){27}}
    \linethickness{1pt}\put(-15,0){\line(1,0){195}}
    \end{picture}
    $$
It satisfies
$\LM(\varphi(D_\nu))=\LM(\varphi(D_1))\cdot\LM(\varphi(D_2))\cdot\LM(\varphi(D_3))=\rho_{\tilde{\nu}_1}\rho_{\tilde{\nu}_2}\rho_{\tilde{\nu}_3}=\rho_\nu.$
\qed
\end{exmp}

To generalize the above example, we need to separate a partition
$\nu$ into substrings $\tilde{\nu}_1,\tilde{\nu}_2,\dots$, each of which
contains at most 3 numbers. Every substring $\tilde{\nu}_j$
corresponds to a $D_j\in\D'$ satisfying
$\LM(\varphi(D_j))=\rho_{\tilde{\nu}_j}$. The correspondence is
specified in table (\ref{table}). Then by putting all $D_j$ together
and adding appropriate extra points if necessary, we obtain $D\in\D$
such that
$$\LM(\varphi(D))=\prod_j\LM(\varphi(D_j))=\prod_j\rho_{\tilde{\nu}_j}=\rho_\nu.$$

\subsection{Proof of the main theorem.}
The following crucial lemma provides an effective method to verify
if a set of alternating polynomials is linearly independent by using
$\varphi$.

\begin{lem}\label{lem:bar phi}
  Fix  $(d_1,d_2)$. Let  $f\in\C[x_1,y_1,\dots,x_n,y_n]^\epsilon$ be
  a bi-homogeneous alternating polynomial of bi-degree $(d_1,d_2)$.
  If $\varphi(f)\neq 0$, then $f\not\equiv 0$
  modulo lower degrees.
  As a consequence, $\varphi$ induces a well-defined linear map
$$
\bar{\varphi}: M_{d_1,d_2} \longrightarrow \C[\rho_1, \rho_2,...]_k.
$$
\end{lem}
\begin{proof}
Suppose $\varphi(f)\neq 0$. By Proposition \ref{prop:varphi and
Delta}, after replacing $n$ by a sufficiently large integer if
necessary, we can assume that $f$ is linearly equivalent to
$\sum_\mu a_\mu F_\mu$ modulo lower degrees, where $F_\mu$ are
special minimal staircase forms. Since $\varphi(f)\neq 0$,
Proposition \ref{prop:varphi and Delta} guarantees $a_\mu\neq 0$ for
some $\mu$. Using the fact that $\{\Delta(F_\mu)\}_{\mu}$ are
linearly independent in $M_{d_1,d_2}$, we conclude that $f\not\equiv
0$ modulo lower degrees.
\end{proof}

The map $\bar{\varphi}$ is natural and useful in the study of
$M_{d_1,d_2}$. Our main theorem (Theorem \ref{main:k<=n-3}) implies
that, for $k:={n\choose 2}-d_1-d_2\le n-3$ $(d_2\le d_1)$, the map $\bar{\varphi}$
is injective and the image is spanned by
$\{\rho_\nu\}_{\nu\in\Pi_{d_2,k}}$. For more general $k$, we expect
that the injectivity still holds. All the computations we did so far
support this conjecture.
\begin{conj}\label{conj:bar phi injective}
The linear map $\bar{\varphi}$ is injective.
\end{conj}
In fact, we can show the following.
\begin{prop}
  Conjecture \ref{conj} implies Conjecture \ref{conj:bar phi
  injective}.
\end{prop}
\begin{proof}
  Assume that Conjecture \ref{conj} is true. Suppose $f\in\C[x_1,y_1,\dots,x_n,y_n]^\epsilon$
  is a bi-homogeneous alternating polynomial of bi-degree
  $(d_1,d_2)$  satisfying
  $\bar{\varphi}(f)=0$.

  Conjecture \ref{conj} implies that the elements of
  $\D_n^{catalan}$ with bi-degree $(d_1,d_2)$ form a basis of
  $M_{d_1,d_2}$, so we can express $f$ as a linear combination
  $\sum_i a_i\, \Delta(D_i)$, $D_i\in \D_n^{catalan}$.  Define $D'_i\in\D_{n+\ell}^{catalan}$ as in
  the proof of Proposition \ref{prop:upper bound}. Then
  $$\bar{\varphi}(\sum_i a_i\, \Delta(D'_i))=\bar{\varphi}(\sum_i
  a_i\, \Delta(D_i))=\bar{\varphi}(f)=0.$$ But $\bar{\varphi}: M_{d_1+\ell,d_2}\to \C[\rho_1,\rho_2,\dots]_k$
  is injective (since $k\le (n+\ell)-3$ for sufficiently large
  $\ell$). So $\sum_i a_i\, \Delta(D'_i)=0$, which implies $a_i=0\;\forall
  i$ and therefore $f\equiv\sum_i a_i\, \Delta(D_i)=0$.
\end{proof}

\begin{lem}\label{partitions with one part}
Let $w\ge 2\in\N$. Suppose
$$D=\{P_1,\dots,P_{w+1}\}\in\D'_{w+1},$$
where $P_i$ are all distinct and
$$|P_1|=|P_2|=0,\quad |P_i|=i-2, \quad 3\leq i\leq w+1.$$
Then the leading term $$\LT(\varphi(D))=(|P_1|_y-|P_2|_y)\rho_w.$$
In particular, the leading monomial $$\LM (\varphi(D))=\rho_{w}.$$
\end{lem}
\begin{proof}
  Immediately follows from the definition of $\varphi(D)$.
\end{proof}
\begin{lem}\label{partitions with two parts}
Let $v,w\in\N$ and $2\leq v\leq w$. Suppose
$$D=\{P_1,...,P_{w+2}\}\in\D'_{w+2},$$
where $P_i$ are all distinct and
$$
|P_i|=\left\{\begin{array}{ll} 0, & \text{ if\; }i=1, 2;\\
i-2, & \text{ if\; }3\leq i\leq w-v+3;\\
i-3, & \text{ if\; }w-v+4\leq i\leq w+2.
\end{array} \right.
$$
Then the leading term $$\LT(\varphi(D))=-(|P_1|_y-|P_2|_y)(|P_{w-v+3}|_y-|P_{w-v+4}|_y)\rho_v\rho_w.$$
In particular, the leading monomial $$\LM (\varphi(D))=\rho_{v}\rho_{w}.$$
\end{lem}
\begin{exmp}
For $v=2$, $w=3$,
$D=\{(-1,1),(0,0),(0,1),(0,2),(1,1)\}$.
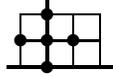
$$
\begin{picture}(50,25)
\put(-3,-3){$\bullet$}\put(-13,7){$\bullet$}\put(-3,7){$\bullet$}\put(7,7){$\bullet$}
\put(-3,17){$\bullet$}
\boxs{-10,0}\boxs{0,0}\boxs{10,0}
\boxs{-10,10}\boxs{0,10}\boxs{10,10}
\linethickness{1pt}\put(0,0){\line(0,1){27}}
\linethickness{1pt}\put(-15,0){\line(1,0){40}}
\end{picture}
$$
A simple computation shows that
$$\varphi(D)=-\rho_2\rho_3+\rho_1\rho_4+\rho_1\rho_2^2-2\rho_1^2\rho_3+2\rho_1^3\rho_2-\rho_1^5,$$
so the $\LT(\varphi(D))=-(1-0)(2-1)\rho_2\rho_3=-\rho_2\rho_3$ as asserted in the above lemma.
\end{exmp}

\begin{proof}[Proof of Lemma \ref{partitions with two parts}] Suppose
$\varphi(D)=\sum a_\mu \rho_\mu$. First we show that $a_\mu\neq0$
implies $\rho_\mu\le\rho_v\rho_w$. Suppose $a_\mu\neq 0$. There exist $\sigma\in S_{w+2}$ and integers $\{w^{(i)}_j\}$ such that  the summand
$$\left(\text{sgn}(\sigma) \prod_{i=1}^n
\rho_{w^{(i)}_1}\rho_{w^{(i)}_2}\cdots\rho_{w^{(i)}_{b_i}}\right)
$$
in (\ref{varphi_version2}) is not zero, and
\begin{equation}\label{summand}
\rho_\mu=\prod_{i=1}^n
\rho_{w^{(i)}_1}\rho_{w^{(i)}_2}\cdots\rho_{w^{(i)}_{b_i}}.
\end{equation}
Because of condition (\ref{varphi condition}), we must have
$$\sigma(i)- 1-a_i-b_i\ge 0,\quad\forall i\in[1,w+2],$$ in particular, $$\sigma(w-v+3)\ge
w-v+2,\quad \sigma(w-v+4)\ge w-v+2.$$ Since $\sigma$ is a
permutation, $\sigma(w-v+3)$ and $\sigma(w-v+4)$ are different
from each other, hence at least one of them is greater than or equal to
$w-v+3$. Let $u$ be $w-v+3$ or $w-v+4$ such that
$\sigma(u)\ge w-v+3$. Since $\sigma(u)\leq w+2$ and $|P_u|(=a_u+b_u)=w-v+1$, we have
$$1\le\sigma(u)-1-|P_u|\le v.$$
By condition (\ref{varphi condition}),
$$w^{(u)}_1+...+w^{(u)}_{b_u}= \sigma(u)-1-a_u-b_u\in[1,v],$$
Take $j\in\N^+$, $1\le j\le b_u$ such that $w^{(u)}_j\neq 0$, then $\rho_{w^{(u)}_j}$ is a factor of $\rho_\mu$ by (\ref{summand}).
Since $w^{(u)}_j\le v\le w$, $\rho_\mu\le \rho_v\rho_w$.
Therefore $a_\mu\neq 0$ implies
$\rho_\mu\le\rho_v\rho_w$.

Now we show that $a_{\mu}\neq 0$ for $\mu=(v,w)$. Assume the monomial $\rho_v\rho_w$ appears in (\ref{summand}).  By the above argument,
it is necessary that
$$\sigma(u)-1-|P_u|=v,$$
which implies $\sigma(u)=w+2$. Denote $\delta=u-(w-v+3)\in\{0,1\}$. On the other hand, since $\sigma(1)$ and $\sigma(2)$ cannot be 1,
we may assume $\sigma(1+\epsilon)\neq 1$ for $\epsilon\in\{0,1\}$. Then $\sigma(1+\epsilon)-1-|P_{1+\epsilon}|=w$, hence
$\sigma(1+\epsilon)=w+1$. For every positive integer $i\le w+2$ that $i\neq 1+\epsilon,\, i\neq u$, we must have $\sigma(i)=1+|P_i|$.
So $\sigma\in S_n$ must be one of the following.
$$ \sigma(i)=\left\{\begin{array}{ll}
1,& \text{ if }i=2-\epsilon; \\
w+1, &  \text{ if } i=1+\epsilon; \\
i-1, & \text{ if } \epsilon+2\leq i\leq w-v+2+\delta; \\
w+2, &  \text{ if } i=w-v+3+\delta; \\
i-2, & \text{ if } w-v+4+\delta \leq i\leq w+2,
\end{array}\right.$$
for $(\epsilon,\delta)=(0,0)$, $(0,1)$, $(1,0)$ or $(1,1)$. By routine computation,
\begin{center}
\begin{tabular}{|c|c|c|}
  \hline
  $\epsilon$ & $\delta$ & \quad coefficient of $\rho_{v}\rho_{w}$ corresponding to $\sigma$\quad\quad \\ \hline
  0 & 0 & $-|P_1|_y |P_{w-v+3}|_y$\\ \hline
  0 & 1 & $+|P_1|_y |P_{w-v+4}|_y$\\ \hline
  1 & 0 & $+|P_2|_y |P_{w-v+3}|_y$\\ \hline
  1 & 1 & $-|P_2|_y |P_{w-v+4}|_y$\\ \hline
\end{tabular}.\end{center}
Adding the above 4 coefficients gives $$a_\mu=a_{(v,w)}=
-(|P_1|_y-|P_2|_y)(|P_{w-v+3}|_y-|P_{w-v+4}|_y)\neq0.$$
\end{proof}

\begin{defn}\label{seq_of_substrings}
To any sequence $\nu=\{\nu_1\le\nu_2\le\cdots\le\nu_n\}$ of positive
integers, we associate a sequence
$\tilde{\nu}=\{\tilde{\nu}_i \}$ of subsequences of $\nu$, each subsequence has the specified number of elements as follows.
Denote by $c$ the number of 1's in $\nu$ and $m:=n-c$.
$$\underbrace{1,1,1}_{3};\cdots;\underbrace{1,1,1}_{3};
\underbrace{1,\dots,1}_{1,2\hbox{ or }3};
\underbrace{\nu_{c+1},\nu_{c+2}}_2;
\cdots;
\underbrace{\nu_{c+2\lceil\frac{m}{2}\rceil-3},\dots,\nu_{c+2\lceil\frac{m}{2}\rceil-2}}_2,
\underbrace{\nu_{c+2\lceil\frac{m}{2}\rceil-1},\dots,\nu_{c+m}}_{1\hbox{ or }2}.
$$
To be precise,
$$
\tilde{\nu}_i=\left\{\begin{array}{ll}
        (1,1,1),& 1\le i\le \lceil\frac{c}{3}\rceil-1; \\
        (\underbrace{1,\dots,1}_{c+3-3\lceil\frac{c}{3}\rceil}), &i=\lceil\frac{c}{3}\rceil;\\
        (\nu_{c+2(i-\lceil\frac{c}{3}\rceil)-1},\nu_{c+2(i-\lceil\frac{c}{3}\rceil)}),& \lceil\frac{c}{3}\rceil+1\le i\le \lceil\frac{c}{3}\rceil+\lceil\frac{m}{2}\rceil-1; \\
        (\nu_{c+2\lceil\frac{m}{2}\rceil-1},\dots,\nu_{c+m}),&i=\lceil\frac{c}{3}\rceil+\lceil\frac{m}{2}\rceil.\\
              \end{array}
 \right.
$$
\end{defn}

\begin{exmp}
If $\nu=(9)$ then $\tilde{\nu}=((9))$.

If $\nu=(1,1,1,1)$ then $\tilde{\nu}=((1,1,1),(1))$.

If $\nu=(1,1,1,1,10)$ then $\tilde{\nu}=((1,1,1),(1),(10))$.

If $\nu=(1,1,1,1,1,1,1,1,3,3,5,5)$ then
$\tilde{\nu}=((1,1,1),(1,1,1),(1,1),(3,3),(5,5))$.

If $\nu=(1,1,1,2,2,2,3,3,7,7)$ then
$\tilde{\nu}=((1,1,1),(2,2),(2,3),(3,7),(7))$.
\end{exmp}

\begin{proof}[Proof of Theorem \ref{main:k<=n-4}] The following
table is the building block of our proof. In the table below,
$|\mu|$ denotes the sum of all numbers in $\mu$.
\begin{equation}\label{table}
\begin{tabular}{|c|l|c|c|}
\hline
$\mu$ & $E_\mu\in\D'$ &$|\mu|$ &$\#E_\mu$\\
\hline (1,1,1) & $\{P_1,P_2,P_3\}, \; |P_1|=|P_2|=|P_3|=0$ & 3&
3\\
\hline (1,1) & $\{P_1,P_2,P_3,P_4\}, \; |P_1|=|P_2|=0,|P_3|=|P_4|=2$
&2&4\\
\hline (1) &
$\{P_1,P_2\}, \; |P_1|=|P_2|=0$ &1&2\\
\hline $\begin{array}{c}(v,w)\\ 2\le v\le w\end{array}$ &
$\begin{array}{l}
\{P_1,\dots,P_{w+2}\}\in\D'_{w+2}, \hbox{ such that }\\
|P_i|=\left\{\begin{array}{ll}
0, & \text{ if\; }1\leq i\leq 2;\\
i-2, & \text{ if\; }3\leq i\leq w-v+3;\\
i-3, & \text{ if\; }w-v+4\leq i\leq w+2.
\end{array} \right.
\end{array}
$
&$v+w$&$w+2$
\\
\hline $\begin{array}{c}(w)\\ w\geq 2\end{array}$
& $\begin{array}{l}\{P_1,\dots,P_{w+1}\}, \hbox{ such that }\\
|P_1|=|P_2|=0, |P_i|=i-2 \,\,(3\leq i\leq w+1)\end{array}$
&$w$&$w+1$\\
\hline
\end{tabular}
\end{equation}

We claim that, in the above table, the leading monomial
$\LM(\varphi(E_\mu))=\rho_\mu$. Indeed, the case $\mu=(1,1,1)$ or
$(1)$ follows from Lemma \ref{varphi_computation_lemma} (v); the
case $\mu=(1,1)$ follows from Lemma \ref{varphi_computation_lemma}
(iv)(v); the case $\mu=(v,w)$ follows from Lemma \ref{partitions
with two parts}; the case $\mu=(w)$ follows from Lemma
\ref{partitions with one part}.

Let $\tilde{\nu}=\{\tilde{\nu}_1,\dots,\tilde{\nu}_m\}$ be defined
as in Definition \ref{seq_of_substrings}. The idea of the
construction of $D_\nu$ is to take the union of translations of
$E_{\tilde{\nu}_1}, \dots, E_{\tilde{\nu}_m}$ together with some
points in $\N\times\N$ that do not affect the value of $\varphi$.

 We consider 2 cases separately.

CASE 1: $\tilde{\nu}_m\neq (1,1,1)$.

Define translating vectors $T_1,\dots,T_m\in\N\times \N$ as follows.
$T_m=(1,0)$,
$$T_i=(1+\#E_{\tilde{\nu}_{i+1}}+\#E_{\tilde{\nu}_{i+2}}+\cdots+\#E_{\tilde{\nu}_{m}},0),\quad \forall i\in [1,m-1].$$
Define
$$n_0=1+\#E_{\tilde{\nu}_{1}}+\#E_{\tilde{\nu}_{2}}+\cdots+\#E_{\tilde{\nu}_{m}}.$$
Then $n_0\le (1+|\tilde{\nu}_{1}|+|\tilde{\nu}_{2}|+\cdots+|\tilde{\nu}_{m}|)+3=k+4\le n$. 
Choose $P_j\in\N\times\N$ such that $|P_j|=j-1$ for $j\in
[n_0+1,n]$. Define $D\in\D'$ as follows,
\begin{equation}\label{eq:D}
D=\{(0,0)\}\;\cup \; \bigcup_{i=1}^m(E_{\tilde{\nu}_i}+T_i)\;\cup
\bigcup_{j=n_0+1}^n \{P_j\},
\end{equation}
where $E_{\tilde{\nu}_i}+T_i$ denotes the set of points in $\N\times
\N$ obtained by adding each point in $E_{\tilde{\nu}_i}$ by the
translating vector $T_i$.

Now we prove the following claim.

\noindent \emph{Claim.} Fix $\nu\in\Pi_{d_2,k}$. For any integer
$d'_2$ satisfying $\#\nu\le d'_2\le {n\choose 2}-k-(\#\nu)$, define
$d'_1={n\choose 2}-k-d'_2$. Then we can make choices of
$E_{\tilde{\nu}_i}$ and $P_j$ in (\ref{eq:D}), such that the
bi-degree of $D$ is $(d'_1,d'_2)$, and the $x$-coordinates of the
points in $D$ are non-negative, i.e. $D\in\D$.

\begin{proof}[Proof of Claim]
We give the exact lower bound and upper bound for the $y$-degree of
$D$, and shows that any integers between the lower bound and upper
bound can be the $y$-degree of some $D$.

For the exact lower bound, we want to construct $P_j$ and
$E_{\tilde{\nu}_i}$ such that their $y$-degrees are as small as
possible. We let $P_j=(j-1,0)$ and $E_{\tilde{\nu}_i}$ be as
follows:
$$\begin{tabular}{|c|l|c|}
\hline ${\tilde{\nu}_i}$ & $E_{\tilde{\nu}_i}\in\D'$ & $y$-degree of $E_{\tilde{\nu}_i}$\\
\hline (1,1,1) & $\{(-2,2),(-1,1),(0,0)\}$&3\\
\hline (1,1) & $\{(-1,1),(0,0),(1,1),(0,2)\}$&2\\
\hline (1) & $\{(-1,1),(0,0)\}$&1\\
\hline $\begin{array}{c}(v,w)\\ 2\le v\le w\end{array}$ &
$\{(0,0),(1,0),\dots,(w-1,0)\}\cup\{(-1,1),(w-v,1)\}$
&2\\
\hline $\begin{array}{c}(w)\\ w\geq 2\end{array}$
&$\{(-1,1),(0,0),(1,0),\dots,(w-1,0)\}$&1\\
\hline
\end{tabular}$$
and denote the resulting $D$ by $D_{\min y}$. Observe that the
$y$-degree of $E_{\tilde{\nu}_i}$ is equal to $\#{\tilde{\nu}_i}$
for all ${\tilde{\nu}_i}$ in the table, so the $y$-degree of
$D_{\min y}$ is $\sum_{i=1}^m(\#\tilde{\nu}_i)=(\#\nu)$.

For the exact upper bound, we need only to note that if $D\in\D_n$
can be constructed as (\ref{eq:D}), then the transpose of $D$ (i.e.
swap the $x$ and $y$ coordinates of each point in $D$) can also be
constructed as (\ref{eq:D}) for some choices of $P_j$ and
$E_{\tilde{\nu}_i}$. In particular, the transpose of $D_{\min y}$,
denoted by $D_{\max y}$, can be constructed as (\ref{eq:D}). The
$y$-degree of $D_{\max y}$ is ${n\choose 2}-k-(\#\nu)$, and is the
maximal $y$-degree for all possible $D\in\D_n$ constructed as
(\ref{eq:D}).

Finally, by moving an appropriate  point of $D$ to the north-west
direction, the $y$-degree increases by $1$, so every integer between
$\#\nu$ and ${n\choose 2}-k-(\#\nu)$ is the $y$-degree of some $D$.
This completes the proof of Claim.
\end{proof}

Now by assumption $d_2\le d_1$, $d_1+d_2={n\choose 2}-k$, and
$(\#\nu)\le d_2$ since $\nu$ is a partition of $k$ into no more than
$d_2$ parts. Therefore $(\#\nu)\le d_2\le{n\choose 2}-k-(\#\nu)$ and
 by the above claim $d_2$ is the $y$-degree of some $D\in\D$
constructed as (\ref{eq:D}). Take this $D$ and denote it by $D_\nu$.
The bi-degree of $D_\nu$ is $(d_1,d_2)$. Applying Lemma
\ref{varphi_computation_lemma} (ii)(iii)(iv),
$$\varphi(D_\nu)=\prod_{i=1}^m\varphi(E_{\tilde{\nu}_i}),$$
hence by Lemma \ref{product of lowest terms} (c),
$$\LM(\varphi(D_\nu))=\prod_{i=1}^m\LM(\varphi(E_{\tilde{\nu}_i}))=\prod_{i=1}^m\rho_{\tilde{\nu}_i}=\rho_\nu.$$

CASE 2: $\tilde{\nu}_m= (1,1,1)$.

In this case, $(\#\nu)=k=3m$. Choose $D\in\D$ to satisfy:
$|P_j|=j-1$ for $1\le j\le n-3m$,
$|P_{n-3m+3j-2}|=|P_{n-3m+3j-1}|=|P_{n-3m+3j}|=n=3m+3j-3$. By
assumption, $\nu\in\Pi_{d_2, k}$, so $d_2\ge k$ in this case. It is
straightforward to verify that we can choose such a $D$ to have
bi-degree $(d_1,d_2)$. This completes the proof of Theorem
\ref{main:k<=n-4}.
\end{proof}

\begin{proof}[Proof of Theorem \ref{main:k<=n-3}]
The proof is almost identical with the one of Theorem
\ref{main:k<=n-4}. We only need to modify the row $\mu=(1,1)$ in the
table (\ref{table}). Instead of using $E_{(1,1)}\in\D'$ (which
contains 4 points), we use two elements $E_{(1,1)}'$ and
$E_{(1,1)}''$ in $D'$, each of which contains 3 points.
$$E_{(1,1)}'=\{(-a-1,a+1),(-a,a),(a+1,a)\},$$
$$E_{(1,1)}''=\{(-a-1,a+1),(-a,a),(a,a+1)\}.$$
A simple computation shows
$$\varphi(E_{(1,1)}')=\rho_2,\quad \varphi(E_{(1,1)}'')=-\rho_1^2+\rho_2,$$
so
$$\varphi(E_{(1,1)}')-\varphi(E_{(1,1)}'')=\rho_1^2.$$
Here we need to be cautious that the bi-degree of $E_{(1,1)}'$ and
$E_{(1,1)}''$ are not the same. This will not bring any problem,
since we can move points in other $E_\mu$ to adjust the total
bi-degree. Eventually, supposing thata $\ell$ is the integer that
$\tilde{\nu}_\ell=(1,1)$, we can construct $D'_\nu,D''_\nu\in\D$
both of bi-degree $(d_1,d_2)$ such that
$$\varphi(D'_\nu)=\varphi(E'_{(1,1)})\prod_{i\neq \ell}\varphi(E_{\tilde{\nu}_i}),\quad \varphi(D''_\nu)=\varphi(E''_{(1,1)})\prod_{i\neq \ell}\varphi(E_{\tilde{\nu}_i}).$$
Then $f:=\Delta(D'_\nu)-\Delta(D''_\nu)$ satisfies
$\LM(\varphi(f))=\rho_\nu$.

Now for each $\nu\in\Pi_{d_2,k}$, we can construct $f_\nu$ such that
$\LM(\varphi(f))=\rho_\nu$. If we write down the coefficient matrix
for $\varphi(f_\nu)$ with basis $\{\rho_\mu\}_{\mu\in\Pi_k}$
arranged in decreasing order, we obtain a row echelon form with rank
$p(d_2,k)$. So $\dim  M_{d_1, d_2}\ge p(d_2, k)$ by Lemma
\ref{lem:bar phi}. Combining the upper bound obtained in Proposition
\ref{prop:upper bound}, we conclude that $\dim  M_{d_1, d_2}= p(d_2,
k)$.
\end{proof}

\section{The condition for the equality $\dim M_{d_1, d_2}=
p(d_2, k)$ to hold}

In Proposition \ref{prop:upper bound} we showed the inequality $\dim
M_{d_1, d_2}\le p(d_2, k)$, then in Theorem \ref{main:k<=n-3} we
showed that ``='' holds for $k\le n-3$. In this section, we show that the
condition $k\le n-3$ is the best we can hope, in the sense of the
following theorem.

\begin{thm}\label{thm:equality holds}
Assume $d_2\leq d_1$. Then
$\dim M_{d_1, d_2}\le p(d_2, k)$, and the equality holds if and only if ``$k\leq n-3$'', or
``$k=n-2$ and $d_2=1$'', or ``$d_2=0$''.
\end{thm}


\begin{proof}
The inequality is proved in Proposition \ref{prop:upper bound}. Then we verify the equality $\dim M_{d_1, d_2}= p(d_2, k)$ in the
specified 3 cases. The case $d_2=0$ is trivial since by definition
$p(0,k)=0$ for $k\ge 1$ and $p(0,0)=1$, we can check the equality
directly. In the case $k=n-2$ and $d_2=1$, $\dim M_{d_1,d_2}=1$
because $\Delta(\{(0,0),(0,1),(1,0),(2,0),\dots,(n-2,0)\})$ forms a
basis for $M_{d_1,d_2}$. The case $k\le n-3$ is proved in Theorem
\ref{main:k<=n-3}.

Now assume $d_2\ge 2$. We use the notation $M^{(n)}_{d_1,d_2}$ to specify which $n$ we are considering.
By Proposition \ref{prop:upper bound}, it suffices to show that $\dim M^{(n)}_{d_1,d_2}<\dim M^{(n+1)}_{d_1+n,d_2}$ for $k=n-2$.

Using the condition $d_1\ge d_2$, it is easy to check that $(d_1-n+3)\ge 0$. So both $(d_1-n+3)$ and $(d_2-2)$ are non-negative integers and $(d_1-n+3)+(d_2-2)={n\choose 2}-k-n+1={n-2\choose 2}$. We know that
$\dim M^{(n-2)}_{d_1-n+3,d_2-2}=1$. Let $D^{(n-2)}\in\D^{Catalan}_{n-2}$ be of bi-degree $(d_1-n+3,d_2-2)$. Define
$$D^{(n+1)}=\{(0,0),(1,0),(0,2)\}\cup \big{(}D^{(n-2)}+(2,0)\big{)}.$$
Then $D^{(n+1)}\in \D^{Catalan}_{n+1}$ is of bi-degree $(d_1+n,d_2)$. On the other hand, every $D^{(n)}\in \D^{Catalan}_n$ of bi-degree $(d_1,d_2)$ determines an element
$$\{(0,0)\}\cup\big{(}D^{(n)}+(1,0)\big{)}$$
in $\D^{Catalan}_{n+1}$ of bi-degree $(d_1+n,d_2)$, which is distinct from $D^{(n+1)}$. Therefore
$$\dim M^{(n)}_{d_1,d_2}<\dim M^{(n+1)}_{d_1+n,d_2}.$$
\end{proof}
\begin{remark} For $k={n\choose 2}-d_1-d_2=0$, we know $\dim M^{(n)}_{d_1,d_2}=p(d_2,k)=1$. We give a straightforward construction of the element $D\in\D^{Catalan}_n$ of bi-degree $(d_1,d_2)$ as follows.
Let $u=\lfloor(2n+1-\sqrt{4n^2-4n-8d_1+9}\,)/2\rfloor$, $i=nu-u(u+1)/2-d_1$,
define $$(x_1,\dots,x_n)=\underbrace{(u-1,u-1,\dots,u-1}_i, \underbrace{u,u,\dots,u}_{n-u-i}, u-1,u-2,\dots,1,0), $$
$$y_i=\#\{j\,|\, i<j, x_j-x_i\in\{0,1\}\}, \quad i=1,\dots,n.$$
Then $D=\{(x_1,y_1),(x_2,y_2),\dots,(x_n,y_n)\}$. \qed
\end{remark}

\newpage
\section{Appendix}
\subsection{Table of the $q,t$ Catalan number for $n=7$}

The number located at the $(i,j)$-th coordinate is equal to the coefficient of $q^i t^j$ in $C_7(q,t)$. The left at the bottom is the $(0,0)$-th position.
$$\aligned
&1    \\
&0 \,\,\,\, 1  \\
&0 \,\,\,\, 1 \,\,\,\, 1 \\
&0 \,\,\,\, 1 \,\,\,\, 1 \,\,\,\, 1  \\
&0 \,\,\,\, 1 \,\,\,\, 2 \,\,\,\, 1 \,\,\,\, 1  \\
&0 \,\,\,\, 1 \,\,\,\, 2 \,\,\,\, 2 \,\,\,\, 1 \,\,\,\, 1 \\
&0 \,\,\,\, 1 \,\,\,\, 3 \,\,\,\, 3 \,\,\,\, 2 \,\,\,\, 1 \,\,\,\, 1 \\
&0 \,\,\,\, \,\,\, \,\,\,\, 2 \,\,\,\, 4 \,\,\,\, 3 \,\,\,\, 2 \,\,\,\, 1  \,\,\,\, 1 \\
&0 \,\,\,\, \,\,\, \,\,\,\, 2 \,\,\,\, 4 \,\,\,\, 5 \,\,\,\, 3 \,\,\,\, 2  \,\,\,\, 1  \,\,\,\, 1 \\
&0 \,\,\,\, \,\,\, \,\,\,\, 1 \,\,\,\, 4 \,\,\,\, 5 \,\,\,\, 5 \,\,\,\, 3  \,\,\,\, 2  \,\,\,\, 1 \,\,\,\, 1 \\
&0 \,\,\,\, \,\,\, \,\,\,\, 1 \,\,\,\, 3 \,\,\,\, 6 \,\,\,\, 6 \,\,\,\, 5  \,\,\,\, 3  \,\,\,\, 2 \,\,\,\, 1  \,\,\,\, 1  \\
&0 \,\,\,\, \,\,\, \,\,\,\, \,\,\, \,\,\,\, 2 \,\,\,\, 5 \,\,\,\, 7 \,\,\,\, 6 \,\,\,\, 5  \,\,\,\, 3  \,\,\,\, 2 \,\,\,\, 1  \,\,\,\, 1  \\
&0 \,\,\,\, \,\,\, \,\,\,\, \,\,\, \,\,\,\, 1 \,\,\,\, 4 \,\,\,\, 6 \,\,\,\, 8 \,\,\,\, 6 \,\,\,\, 5  \,\,\,\, 3  \,\,\,\, 2 \,\,\,\, 1  \,\,\,\, 1  \\
&0 \,\,\,\, \,\,\, \,\,\,\, \,\,\, \,\,\,\, \,\,\, \,\,\,\, 2 \,\,\,\, 5 \,\,\,\, 7 \,\,\,\, 8 \,\,\,\, 6 \,\,\,\, 5  \,\,\,\, 3  \,\,\,\, 2 \,\,\,\, 1  \,\,\,\, 1  \\
&0 \,\,\,\, \,\,\, \,\,\,\, \,\,\, \,\,\,\, \,\,\, \,\,\,\, 1 \,\,\,\, 3 \,\,\,\, 6 \,\,\,\, 8 \,\,\,\, 8 \,\,\,\, 6 \,\,\,\, 5  \,\,\,\, 3  \,\,\,\, 2 \,\,\,\, 1  \,\,\,\, 1  \\
&0 \,\,\,\, \,\,\, \,\,\,\, \,\,\, \,\,\,\, \,\,\, \,\,\,\, \,\,\, \,\,\,\, 1 \,\,\,\, 3 \,\,\,\, 6 \,\,\,\, 7 \,\,\,\, 8 \,\,\,\, 6 \,\,\,\, 5  \,\,\,\, 3  \,\,\,\, 2 \,\,\,\, 1  \,\,\,\, 1  \\
&0 \,\,\,\, \,\,\, \,\,\,\, \,\,\, \,\,\,\, \,\,\, \,\,\,\, \,\,\, \,\,\,\, \,\,\, \,\,\,\, 1 \,\,\,\, 3 \,\,\,\, 5 \,\,\,\, 6 \,\,\,\, 7 \,\,\,\, 6 \,\,\,\, 5  \,\,\,\, 3  \,\,\,\, 2 \,\,\,\, 1  \,\,\,\, 1  \\
&0 \,\,\,\, \,\,\, \,\,\,\, \,\,\, \,\,\,\, \,\,\, \,\,\,\, \,\,\, \,\,\,\, \,\,\, \,\,\,\, \,\,\, \,\,\,\, 1 \,\,\,\, 2 \,\,\,\, 4 \,\,\,\, 5 \,\,\,\, 6 \,\,\,\, 5 \,\,\,\, 5  \,\,\,\, 3  \,\,\,\, 2 \,\,\,\, 1  \,\,\,\, 1  \\
&0 \,\,\,\, \,\,\, \,\,\,\, \,\,\, \,\,\,\, \,\,\, \,\,\,\, \,\,\, \,\,\,\, \,\,\, \,\,\,\, \,\,\, \,\,\,\, \,\,\, \,\,\,\, \,\,\, \,\,\,\, 1 \,\,\,\, 2 \,\,\,\, 3 \,\,\,\, 4 \,\,\,\, 4 \,\,\,\, 4  \,\,\,\, 3  \,\,\,\, 2 \,\,\,\, 1  \,\,\,\, 1  \\
&0 \,\,\,\, \,\,\, \,\,\,\, \,\,\, \,\,\,\, \,\,\, \,\,\,\, \,\,\, \,\,\,\, \,\,\, \,\,\,\, \,\,\, \,\,\,\, \,\,\, \,\,\,\, \,\,\, \,\,\,\, \,\,\, \,\,\,\, \,\,\, \,\,\,\, 1 \,\,\,\, 1 \,\,\,\, 2 \,\,\,\, 2 \,\,\,\, 3 \,\,\,\, 2  \,\,\,\, 2 \,\,\,\, 1  \,\,\,\, 1  \\
&0 \,\,\,\, \,\,\, \,\,\,\, \,\,\, \,\,\,\, \,\,\, \,\,\,\, \,\,\, \,\,\,\, \,\,\, \,\,\,\, \,\,\, \,\,\,\, \,\,\, \,\,\,\, \,\,\, \,\,\,\, \,\,\, \,\,\,\, \,\,\, \,\,\,\, \,\,\, \,\,\,\, \,\,\, \,\,\,\, \,\,\, \,\,\,\, \,\,\, \,\,\,\, 1 \,\,\,\, 1 \,\,\,\, 1 \,\,\,\, 1 \,\,\,\, 1 \,\,\,\, 1 \\
&0 \,\,\,\,\, 0 \,\,\,\, 0 \,\,\,\, 0 \,\,\,\, 0 \,\,\,\, 0 \,\,\,\,
0 \,\,\,\, 0 \,\,\,\, 0 \,\,\,\, 0 \,\,\,\, 0 \,\,\,\, 0 \,\,\,\, 0
\,\,\,\, 0 \,\,\,\, 0 \,\,\,\, 0 \,\,\,\, 0 \,\,\,\, 0 \,\,\,\, 0
\,\,\,\, 0 \,\,\,\, 0 \,\,\,\, 1
\endaligned$$

Tables for $n\leq 7$ can be found at F. Bergeron's website

\begin{verbatim}
http://bergeron.math.uqam.ca/n_fact_Conjecture/qt_catalan.pdf
\end{verbatim}

\subsection{Macaulay 2 code for computing $\varphi$} For the convenience of the reader, we provide a Macaulay 2 code \cite{M2} for computing $\varphi$ (function ``phi'' in
the code) defined in \S2.1, together with an example of computation
$$\varphi(\{(-1,1),(0,0),(0,1),(0,2),(1,1)\}).$$
\begin{verbatim}
i1 : phi=(D)->(
       local R,n,k,sgn,s,total,t,bi,sumrho,prod;
       n=#D;
       R=ZZ[r_0..r_n];
       k=n*(n-1)//2; for i from 1 to n do
          k=k-D#(i-1)#0-D#(i-1)#1;
       total=0;
       scan(permutations(n),sigma->(
         sgn=1;
         for i from 0 to n do
           for j from i+1 to #sigma-1 do
            (if sigma#i>sigma#j then sgn=sgn*(-1));
         t=1;
         for i from 1 to n do
          (
            ai=D#(i-1)#0;bi=D#(i-1)#1;
            s=(sigma#(i-1)+1)-1-ai-bi;
            if (s<0) then (t=0;break) else
            if (bi==0) and (s>0) then (t=0;break) else
            if (bi==0) and (s==0) then (t=t*1) else
            if (bi==1) then (t=t*r_s) else
             (sumrho=0;
              scan(subsets(s+bi-1,bi-1),su->(
                    prod=r_(su#0);
                    for j from 2 to bi-1 do
                       prod=prod*r_(su#(j-1)-su#(j-2)-1);
                    prod=prod*r_(s+bi-2-su#(bi-2));
                    sumrho=sumrho+prod;     )
                  );
              t=t*sumrho;
             );
          );
          total=total+sgn*t;
         )); --end of scan of sigma.
       return sub((-1)^k*total,{r_0=>1});
     )

o1 = phi

o1 : FunctionClosure

i2 : phi({(-1,1),(0,0),(0,1),(0,2),(1,1)})

        5     3        2     2
o2 = - r  + 2r r  + r r  - 2r r  - r r  + r r
        1     1 2    1 2     1 3    2 3    1 4

o2 : R

\end{verbatim}

\end{document}